\RequirePackage{fix-cm}
\documentclass[smallextended]{svjour3}       
\smartqed  
\usepackage{amsfonts,comment}
\usepackage{amssymb,amsmath,url,bm}
\usepackage{mathrsfs}
\usepackage{hyperref}
\usepackage{graphicx,color,float,epstopdf}
\usepackage{mathrsfs}
\usepackage{pst-blur,pstricks-add}
\usepackage{cases}
\usepackage{algorithm}
\usepackage{algorithmicx}
\usepackage{algpseudocode}
\usepackage{graphics,booktabs,epsfig,subfigure}
\usepackage[numbers,sort&compress]{natbib}
\DeclareMathOperator{\prox}{Prox}
\newtheorem{Def}{Definition}
\newcommand{\h}[1]{\mathbf{#1}}

\newcommand{\nn}{\nonumber}

\newcommand{\supp}{{\text{supp}}}

\begin{document}

\title{Unified Analysis on $L_1$ over $L_2$   Minimization for signal recovery
}


\author{Min Tao        \and
       Xiao-Ping Zhang 
}


\institute{Min Tao \at
              Department of
   Mathematics,  National Key Laboratory for Novel Software Technology, Nanjing University, Nanjing, 210093, China. \\
              \email{taom@nju.edu.cn}           
           \and
          Xiao-Ping Zhang \at
              Department of Electrical, Computer and Biomedical Engineering, Ryerson University,  Toronto, ON M5B 2K3, Canada.
}

\date{Received: date / Accepted: date}

\maketitle

\begin{abstract}
In this paper, we carry out a unified study for $L_1$ over $L_2$  sparsity promoting models, which are widely used in the regime of coherent dictionaries for recovering sparse nonnegative/arbitrary signals. First, we provide a unified theoretical analysis on  the  existence of the global solutions of the constrained and the unconstrained $L_{1}/L_{2}$ models.
Second, we analyze the sparse property of any local minimizer of these $L_{1}/L_{2}$ models which serves as a certificate to rule out
the nonlocal-minimizer stationary solutions.
Third, we derive an analytical solution for the proximal operator of the $L_{1} / L_{2}$ with nonnegative constraint.
 Equipped with this, we apply the alternating direction method of multipliers   to  the unconstrained model with nonnegative constraint in a particular splitting way,
referred to as ADMM$_p^+$. We establish its global convergence to a  d-stationary solution (sharpest stationary) without
 the Kurdyka-\L ojasiewicz assumption.
Extensive numerical simulations  confirm the superior  of ADMM$_p^+$ over
 the state-of-the-art methods in sparse recovery. In particular, ADMM$_p^+$ reduces computational time by about $95\%\sim99\%$ while achieving a much higher accuracy than the commonly used  scaled gradient projection method for the wavelength misalignment problem.
\keywords{Sparse recovery \and fractional programming \and coherent dictionary \and d-stationarity}
\subclass{MSC 90C26 \and MSC    90C90 \and 49N45}
\end{abstract}

Compressive sensing (CS) is to seek the sparsest solution from a set of undersampled linear measurements. Mathematically, a fundamental
 problem in CS can be formulated as a constrained model,

\begin{eqnarray}\label{CSzero}
\min _{{\h x} \in {\cal X}}\|{\h x}\|_{0}, \;\; s.t. \;\; A {\h x}={\h b},
\end{eqnarray}
where $A \in \mathbb{R}^{m \times n}(m \ll n)$ is a sensing matrix and  the observation of ${\h b}\in\mathbb R^n$, $\|\cdot\|_{0}$ is the $L_{0}$ norm (the number of nonzero elements). We consider the recovery of nonnegative/arbitrary compressed signal, which corresponds to ${\cal X}=\mathbb R^n_+$ and $\mathbb R^n$,
 respectively. Unfortunately, the optimization (\ref{CSzero}) is known to be NP-hard \cite{BKNa95}.
One common approach is to relax $L_0$ norm to $L_1$ norm, leading to basis pursuit
model \cite{CDS98}. Theoretically, the exact recovery by the $L_1$ minimization  is
guaranteed under the restricted isometry property \cite{CandesTao05} or
null space property \cite{AWR09}.
Although the $L_{1}$ minimization technique has been widely used, it is not able to reconstruct the sparsest solutions
when columns of $A$ are highly coherent, such as those applications rising from discretization of continuum image problems (such as medical and radar) when the grid spacing is below the Rayleigh threshold \cite{AW12}.

 As such, various nonconvex regularizers, such as the $L_p$ (quasi-)norm ($0<p<1$) \cite{CR07},
  $L_1$-$L_2$ \cite{YEX14}, transform $L_1$ \cite{NIK00}, the ratio of $L_1$ over $L_2$ norm ($L_1$/$L_2$) \cite{YEX14}  have been developed to enhance the recovery quality.
Among these nonconvex regularizations, $L_1/L_2$
 can approximate  $L_{0}$ norm very well when  the domain is without origin, due
   to its being {\it scale-invariant} and {\it parameter-free} as well as
      $L_{0}$ norm.
      For one-sparse signal, the $L_{1} / L_{2}$ is the same as the $L_0$ norm.
 The $L_1/L_2$ arose as sparseness measure \cite{Hoy02,HNRS09} and has attracted a considerable attention due to its wide applications, e.g.,
  nonlinear matrix factorization \cite{4541671} and
   blind deconvolution \cite{JSW12,Audrey15}.

     In this paper, we focus on two  commonly-used  $L_{1} / L_{2}$ models for signal recovery, i.e., the constrained and the penalized/unconstrained models:
\begin{eqnarray}\label{L1o2Con}
\begin{array}{ll}
\min _{\mathbf{x}}& \displaystyle{\frac{\|\mathbf{x}\|_{1}}{\|\mathbf{x}\|_{2}}} \\[0.2cm]
 s.t. & \mathbf{x} \in \mathcal{H}:=\{{\h x} \in{\cal X} \mid A {\h x} = {\h b}\},
 \end{array}
\end{eqnarray}
and
\begin{eqnarray}
\label{L1o2uncon}
\min _{{\h x}\in {\cal X}} F({\h x}):=\gamma \frac{\|{\h x}\|_{1}}{\|{\h x}\|_{2}}+\frac{1}{2}\|A {\h x}-\mathbf{b}\|_{2}^{2},
\end{eqnarray}
where   $A$ and   ${\h b}$ are defined identically as in (\ref{CSzero}).
The penalized/unconstrained model (\ref{L1o2uncon}) can  tackle both noisy and noiseless observations while (\ref{L1o2Con})
can only deal with unnoisy data, it is more meaningful to develop efficient and convergent algorithms to solve (\ref{L1o2uncon}).


First, to recover the signal, we need to solve $L_1$ over $L_2$ minimization models (\ref{L1o2Con}) or (\ref{L1o2uncon}).
The first question  is whether these models are well-defined.
Recently, Zeng et al. \cite{ZengYuPong20} analyze the existence of global optimal solutions of the constrained model (\ref{L1o2Con}) for the case of ${\cal X}={\mathbb R}^n$. However, the nonemptyness of the global solution set of the model of (\ref{L1o2Con})  with ${\cal X}={\mathbb R}_+^n$ and
  the model (\ref{L1o2uncon}) with ${\cal X}={{\mathbb R}^n_+}/{\mathbb R}^n$ (i.e., ${{\mathbb R}^n_+}$ and ${\mathbb R}^n$)   have not been studied.

Many early works focus on developing different optimization algorithms  for solving $L_{1} / L_{2}$ minimization. To name a few,  the scaled gradient projection method (SGPM) \cite{ELX13,YEX14} for (\ref{L1o2uncon}) with ${\cal X}={\mathbb R}_+^n$, as well as the alternating direction method of multiplier (ADMM) approach \cite{RWDL19} for (\ref{L1o2Con}) with ${\cal X}={\mathbb R}^{n}$, accelerated schemes for solving the constrained model
(\ref{L1o2Con}) with ${\cal X}={\mathbb R}^{n}$ \cite{WYYL20},
all lack global convergence guarantees.
There are two exceptions:  Zeng et al. \cite{ZengYuPong20} apply moving-balls-approximation based algorithms to solve $L_1/L_2$ minimization over an inequality constraint for arbitrary signal and prove its local linear convergence under  some  conditions. Our previous work \cite{Tao20} proposes
 an ADMM-based algorithm for solving (\ref{L1o2uncon}) with $\mathcal{X}=\mathbb{R}^{n}$ with global
convergence guarantee.
Although there do exist a few different reformulations for the unconstrained model (\ref{L1o2uncon}) with $\mathcal{X}=\mathbb{R}_+^{n}$, it turns out
that most of these for implementing ADMM result in a  violation of convergence guarantee \cite{LiPong15,HongLuoRazaviyayn16}.
  It is still unknown how to solve the  model (\ref{L1o2uncon})  with convergence guarantee where $\mathcal{X}=\mathbb{R}^{n}_+$.


In this paper, we carry out
 a unified theoretical study on  (\ref{L1o2Con}) and (\ref{L1o2uncon}).
First,   we provide a unified analysis on the existence of global solutions of  (\ref{L1o2Con}) and (\ref{L1o2uncon}).
Inspired by \cite{ZengYuPong20}, we introduce an auxiliary optimization problem and
verify that the solution set of  (\ref{L1o2Con}) is nonempty if  the objective function value of  (\ref{L1o2Con}) is strictly less than
that of the auxiliary optimization problem.
A similar result is also proved for the model (\ref{L1o2uncon}) while the proof is much more complicated.
Then, we illustrate that
 this sufficient condition  can be guaranteed by the $\mu$-spherical section property of ${\cal N}(A)$ (i.e., the null
space of $A$) \cite{Vavasis}.
Second,  we exploit the sparse property of any local minimizer of (\ref{L1o2Con}) or (\ref{L1o2uncon}).
In particular, we prove that any feasible vector cannot be a local minimizer if  the columns of $A$ restricted on the support set  are  linearly
dependent.
Third, we design an efficient and convergent algorithm for the penalized model (\ref{L1o2uncon}) with ${\cal X}={\mathbb R}^n_+$.
To do so, we   derive a closed-form solution for one global solution of the proximal operator of $\left(L_{1} / L_{2}\right)^{+}$ (i.e.,
${\frac{\|\h x\|_1}{\|\h x\|_2}+\iota_{{\mathbb R}^n_+}(\h x)}$ and $\iota_{{\mathbb R}^n_+}(\h x)$ is the indicator function of ${\mathbb R}^n_+$)
 and  accompanied by a practical solver.
 Equipped with this,   we propose a specific  variable-splitting scheme of ADMM for solving
 (\ref{L1o2uncon})  with ${\cal X}={\mathbb R}^n_+$.
We referred to it as   ADMM$_p^+$ by incorporating the practical proximal solver.
  Although there already exist
 many seminal works on convergence analysis for the nonsmooth  nonconvex problem, e.g., \cite{Attouch13,HongLuoRazaviyayn16,LiPong15},
 all these convergence results focus on converging to a stationary point which is much weaker than to a
 d-stationary point \cite{DongTao21,PangRazAlv,LSM20}. Furthermore, all these approaches achieve  global convergence by assuming
 the introduced merit function with the Kurdyka-\L ojasiewicz (KL) property. In contrast,
 we introduce a novel merit function $\cal T$ (see (\ref{shMF})) instead of the augmented Lagrange function or its variants \cite{LiPong15,HongLuoRazaviyayn16} for aiding the global convergence analysis.  Then,  we establish the global convergence of ADMM$_p^+$ converging  to a  d-stationary point
 by proving the merit function with  the KL property.
We conduct extensive experiments on algorithmic behaviors and various sparse recovery model comparisons,
testing on two specific applications. All of these  showcase the superior performance of the proposed approach over the state-of-the-art in sparse nonnegative signal recovery.
In particular, ADMM$_p^+$ always converges to a more accurate solution (d-stationary) and significantly reduces the computational time in comparison with
 SGPM
and accelerated proximal gradient methods (monotone version with fixed stepsize/line search and its nonmonotone versions).
In summary, 	our contributions are threefold:
\begin{enumerate}
	\item[(1)] We provide a unified theoretical analysis on the existence of global solutions of the constrained model (\ref{L1o2Con}) and unconstrained model (\ref{L1o2uncon}).
\item[(2)] We exploit the sparse property of any local minimizer of the constrained model (\ref{L1o2Con}) and unconstrained model (\ref{L1o2uncon}). This property serves as a {\it certificate} to rule out the stationary solution that is not a local minimizer.
	\item[(3)] We derive an analytic solution of the proximal operator of $(L_1/L_2)^+$ which allows us
 to design an efficient algorithm for the unconstrained model (\ref{L1o2uncon}) with ${\cal X}={\mathbb R}^n_+$, i.e., ADMM$_p^+$.
  We establish its global convergence to a d-stationary solution without KL assumption. Extensively numerical simulations further verify the computational efficiency of the ADMM$_p^+$
 over the state-of-the-art in sparse recovery.
\end{enumerate}

%
%
%
%
%

The rest of this paper proceeds as follows. We describe the notations and definitions in Section \ref{Sec-Preliminaries}.
 In Section \ref{wellDef}, we elaborate on the existence of optimal solutions of
(\ref{L1o2Con}) and  (\ref{L1o2uncon}). We analyze the sparse property and provide the exact recovery theory  for  (\ref{L1o2Con})
in Section \ref{RecoveryT}.
In Section \ref{ADMMS}, the proximal operator for $\left(L_{1} / L_{2}\right)^{+}$ is derived. Then, we solve the unconstrained model (\ref{L1o2uncon}) with ${\cal X}={\mathbb R}_+^n$ via ADMM, where its global convergence is established.
Section \ref{NumRes} devotes extensive experiments to showcase the superior performance of the proposed approach in sparse recovery. Conclusions are given in Section \ref{Sec-Conclusion}.

\section{Preliminary} \label{Sec-Preliminaries}

 We use a bold letter to denote a vector, e.g., ${\h x} \in \mathbb{R}^{n}$, and $x_{i}$,  $\|\h x\|_0$ and $|{\h x}|$ denote the $i$-th entry of $\h{x}$, its zero norm of $\h{x}$ and  the vector with the absolute value of $\h{x}$ for
  each entry, respectively. $\|{\h x}\|_2$ and $\|{\h x}\|_p$ $(0<p<1)$ denote its 2-norm and
 $p$-norm ($\|{\h x}\|_p=(\sum_{i=1}^n x_i^p)^{1/p}$), respectively.
 The subscript  $2$ in $\|\cdot\|_2$ is omitted when there is no ambiguity.
  We use $\mathbb R^n_+$, $\mathbb R^n_-$  and $\mathbb R^n_{++}$ to denote the set of  nonnegative, nonpositive  and positive vectors, respectively. The notation of ${{\bf 1}}$ represents a vector with all entries equal to $1$.
   $I_{n}$ is $n \times n$ identity matrix, and $\odot$  presents the componentwise product. We define $[n]:=\{1,2, \ldots, n\}$.
    Given an index set $\mathcal{D} \subseteq [n]$, we use $\sharp(\mathcal{D})$ and $\mathcal{D}^c$ to present the cardinality of $\mathcal{D}$ and its complementary set. We specify that ${\cal X}$ is either ${{\mathbb R}^n_+}$ or ${\mathbb R}^n$ throughout this paper.
For a matrix $A\in\mathbb R^{m\times n}$, we denote $A {\cal X}=\left\{{\h y} \; \mid \; {\h y}=A {\h x},\; {\h x}\in {\cal X}\right\}$
and refer  to the projection onto the closed set of $A \mathcal{X}$ as $\operatorname{Proj}_{(A \mathcal{X})}(\cdot )$.
For a closed set $\mathcal{S} \in \mathbb{R}^{n}$, we use the notation $\iota_{\cal S}(\mathrm{x})$ to represent the indicator function of the set $\mathcal{S}$. Given a matrix $A \in \mathbb{R}^{m \times n}$ or a vector $\h{x} \in \mathbb{R}^{n}$ and an index set $\Lambda \subseteq[n],$ we use $A_{\Lambda}$, ${\h x}|_{\Lambda}$ to denote $A[:, i]_{i \in \Lambda}$ and a subvector of $\h {x}$ with entries in $\Lambda,$ respectively. $\mathcal{N}(A)$ denotes the null space of $A$ and $r(A)$ denotes the rank of $A$.   Given a square matrix $A$,
$A\succ{\bf 0}$ means that $A$ is a positive definite matrix. For a vector ${\h x}$, we use the notation of ${\h x}\not\le{\bf 0}$ to represent
that not all the entries of ${\h x}$ are nonpositive.
Given $\epsilon>0$, we use $\mathcal{B}_{\epsilon}(\hat{\mathbf{x}})$ and $\breve{\mathcal{B}}_{\epsilon}(\hat{\mathbf{x}})$ to denote the open ball of $\{\mathbf{x} \mid\|\mathbf{x}-\hat{\mathbf{x}}\|<\epsilon\}$ and the open ball without the center, respectively. Given two sets of $\mathcal{A}$ and $\mathcal{B},$ we use the notation $\mathcal{A} \backslash \mathcal{B}$ to represent the intersection of $\mathcal{A}$ and the complement of $\mathcal{B}$.  The notation of ${{\mathbb R}^n_+}/{\mathbb R}^n$ means ${{\mathbb R}^n_+}$ or ${\mathbb R}^n$.
%
An extended-real-valued function $f:{\mathbb R}^n\rightarrow (-\infty,+\infty]$ is said to be
 proper if its domain ${\text{dom}} f:=\{{\h x}\;| \;f(\h x)<\infty \}$ is nonempty. A proper function $f$ is said to
 be closed if it is lower semi-continuous.  For a proper closed function $f$ and ${\hat{\h x}}
 \in {\text{dom}} f$,  the regular subdifferential ${\hat {\partial}} f({\hat{\h x}})$
 and the limiting subdifferential $\partial f({\h x})$ [24]  are defined as
$$
		{\hat\partial} f({\hat{\h x}})=\left\{{\h v}\Big |{\mathop{\lim}\limits_{\h x\to{\hat{\h x}}}}\inf_{{\h x}\neq {\hat{\h x}}} \frac{f(\h x)-f(\hat{\h x})-\langle {\h v},{\h x}-{\hat{\h x}}\rangle}{\|{\h x}-\hat{\h x}\|}\ge 0\right\},$$
		$$\partial f({\hat{\h x}}):=\left\{{\h v}\Big | \;\exists\; {\h x}^k\rightarrow {\hat{\h x}},\;f({\h x}^k)\rightarrow f(\hat{\h x}),
 {\h v}^k\in{\hat\partial} f({\h x}^k), {\h v}^k\rightarrow{\h v} \right\},
		$$
respectively.
Note $\partial \iota_{\mathcal{X}}(\mathbf{x})=\{\mathbf{0}\}$ if $\mathcal{X}=\mathbb{R}^{n},$ and $\partial_{\iota \mathcal{X}}(\mathbf{x})=\left\{{\h d} \mid {\h d}_{\Lambda}=\mathbf{0}, {\h d}_{\Lambda^{c}} \leq \mathbf{0},\right.$ $\left. \Lambda={\text{supp}}(\mathbf{x})\right\}$ if
$\mathcal{X}=\mathbb{R}_{+}^{n} .$
Throughout the paper, we assume that ${\h b} \neq 0$ and ${\cal H} \neq  \phi$.
 {\it By defining ${\frac{\|{\bf 0}\|_1}{\|\bf 0\|_2}}=1$,
the objective functions of (\ref{L1o2Con}) and (\ref{L1o2uncon}) are lower semi-continuous over $\cal X$.}
Suppose $f$ be a proper lower semicontinuous function, we define the proximal mapping  \cite[Definition 1.22]{RockWets}: $\prox_f(v) = \arg\min_{x} \Big\{ f(x) + \frac{1}{2}\|x-v\|^2\Big\}.$

For nonconvex and nonsmooth programs, there are various stationary concepts  and d-stationary
 is arguably the sharpest kind among them \cite{DongTao21,PangRazAlv,LSM20}.
A point $\bar{\mathbf{x}}(\neq {\bf 0})\in{\cal X}$ is called a d-stationary point to (\ref{L1o2uncon}), if it satisfies
$
F^{\prime}\left({\bar{\h x}}; {\h x}-{\bar{\h x}}\right) \geq 0 . \quad \forall {\h x} \in {\cal X},
$
where $F^{\prime}\left({\bar{\h x}}; {\h x}-{\bar{\h x}}\right)$ is the directional derivative of $F(\cdot)$.
According to \cite[Definition 2.3.4]{Clark90}
and \cite[Fact 5]{LSM20}, $\bar{\mathbf{x}}(\neq {\bf 0})\in{\cal X}$ is a d-stationary point of $(\ref{L1o2uncon})$  if and only if
\begin{eqnarray}\label{viopt}
\left\langle{\h x}-{\bar{\h x}},\;\gamma\left(\frac{{{\h 1}}}{\|{\bar{\h x}} \|_2}-\frac{\|{\bar{\h x}}\|_{1}}{\|{\bar{\h x}}\|_{2}^{3}} {\bar{\h x}}\right)+A^{\top}(A {\bar{\h x}}-{\h b})\right\rangle \geq 0,\;\;
\forall {\h x} \!\in {\cal X},\!
\end{eqnarray}
where ${{\h 1}} \in {\mathbb R}^n$ by noting that $\frac{\|{\h x}\|_1}{\|{\h x}\|_2}+\iota_{{\mathbb R}_+^n}({\h x})=\frac{{\bf 1}^\top {\h x}}{\|\h x\|_2}+\iota_{{\mathbb R}_+^n}({\h x})$.

Next, we review the concepts of locally sparse set  and the uniformity of a vector \cite{YEX14} and the KL property  \cite{BDL07} which is widely used in convergence  analysis.
\begin{Def}\label{def2.3} ${\h x} \in {\cal H}$ is called locally sparse if $\nexists \;{\h y}\in {\cal H} \backslash\{{\h x}\}$ (${\cal H}$ defined in (\ref{L1o2Con})) such that
$
{\text{supp}}(\h y) \subseteq {\text{supp}}(\h x).
$
Denote by ${\cal H}_{L}=\{\h x \in {\cal H} \mid {\h x}\;\mbox{is locally sparse} \}$.
\end{Def}
\begin{Def} \label{def2.4} The uniformity of ${\h x}$, $ \kappa(\h x)$ is the ratio between the smallest nonzero absolute entry and the largest one in the sense of absolute value, i.e.
$$0<\kappa(\h x):=\frac{\min_{i\in{\text{supp}}(\h x)}{|x_i|}}{\max_{i\in{\text{supp}}(\h x)}{|x_i|}}\le 1.$$
\end{Def}
\begin{Def}\label{def2.1} We say a proper closed function $h: \mathbb{R}^{n} \rightarrow(-\infty,+\infty]$ satisfies KL property at a point $\hat{\h {x}} \in {\text{dom}}\partial h$ if there exist a constant $\alpha \in(0, \infty],$ a neighborhood $U$ of ${\h{\hat x}}$, and a continuous concave function $\phi:[0, \nu) \rightarrow[0, \infty)$ with $\phi(0)=0$ such that
\begin{itemize}
\item[(i)] $\phi$ is continuously differentiable on $(0, \nu)$ with $\phi^{\prime}>0$ on $(0, \nu);$
\item[(ii)] for every $\mathrm{x} \in U$ with $h(\hat{\mathbf{x}})<h(\mathbf{x})<h(\hat{\mathbf{x}})+\nu,$ it holds that
$\phi^{\prime}(h(\mathbf{x})-h(\hat{\mathbf{x}})) {\text{dist}}(\mathbf{0}, \partial h({\h x})) \geq 1.$
\end{itemize}
\end{Def}
Our analysis on the existence of globally optimal solutions is based on the spherical section property (SSP) \cite{Vavasis,ZengYuPong20}.

\begin{Def}\label{def2.5}
 Let $m,\;n$ be two positive integers such that $m<n$. Let $V$ be an $(n-m)$-dimensional subspace of $\mathbb{R}^{n}$ and $\mu$ be a positive integer. We say that $V$ has the $\mu$-spherical section property if
$
\inf _{\mathbf{v} \in V \backslash\{0\}} \frac{\|\mathbf{v}\|_{1}}{\|\mathbf{v}\|_{2}} \geq \sqrt{\frac{m}{\mu}}.
$
\end{Def}

If $A \in \mathbb{R}^{m \times n}(m<n)$ is a random matrix with i.i.d. standard Gaussian entries, then its nullspace has the $\mu$-spherical section property with high probability \cite{ZengYuPong20}.

Consider  the following problem:
\begin{eqnarray}\label{FBmodel} \min_{{\h x}\in\mathbb R^n} f({\h x})=h({\h x})+g({\h x}), \end{eqnarray}
where $h: {\mathbb R}^n\rightarrow \mathbb R$ is $L$-smooth (possibly, nonconvex) and $g:{\mathbb R}^n\rightarrow {\mathbb R}$ is proper
lower semicontinuous function.
The forward-backward algorithm generates the iterate as follows:
\begin{eqnarray}\label{hg}{\h x}^{k+1}\in{\arg\min}_{\h x}\prox_{\alpha g}({\h x}^k-\alpha\nabla h({\h x}^k)),
\end{eqnarray}
where the step size $\alpha\in(0,1/L)$ to guarantee that any accumulation point of the sequence generated above is a stationary point of
(\ref{FBmodel}) \cite{BLR15,LiPong15}.

\section{Existence of optimal solutions}\label{wellDef}
We explore the conditions to guarantee the existence of global solutions of (\ref{L1o2Con}) and (\ref{L1o2uncon}).
 First, we verify that the solution set  of  (\ref{L1o2Con}) is nonempty if the optimal value of (\ref{L1o2Con})
 is less than that of the auxiliary problem. Second, we prove   that the solution set  of (\ref{L1o2uncon}) is nonempty
 when the objective function value of newly-introduced constrained model  is less than the auxiliary problem.
 Finally, we further show that these sufficient conditions can be guaranteed  by the $\mu$-spherical section property of
 the null space of the measurement matrix.
Our analysis is inspired by \cite{ZengYuPong20} and introduce the following auxiliary problem:
\begin{eqnarray}\label{fdccon}
\begin{array}{ll}
f_{dc}^{*}:=&\inf _{{\h d}\in{\cal F}_0} \displaystyle{\frac{\|\mathbf{d}\|_{1}}{\|\mathbf{d}\|_{2}}} \\[0.2cm]
 \mbox{where}   & {\cal F}_{0}:=\{{\h d}\; \mid \; A \mathbf{d}=\mathbf{0}, \;\mathbf{d} \in {\cal X}, \;{\h d} \neq \mathbf{0}\}.
\end{array}
\end{eqnarray}
For analysis convenience, we denote the optimal value of (\ref{L1o2Con})  as $f_{pc}^{*}$ and $f_{p c}^{*}<+\infty$.
 We recall  $\left\{{\h x}^{k}\right\}$ is a {\it minimizing sequence} of (\ref{L1o2Con}) if ${\h x}^{k} \in {\cal H}$ for each $k$ and
$
\lim _{k \rightarrow \infty} \frac{\left\|{\h x}^{k}\right\|_{1}}{\left\|{\h x}^{k}\right\|_{2}}=f_{pc}^{*}$. Therefore we only need to characterize the existence of unbounded minimizing sequence.
 Before that, we provide a sufficient condition to guarantee the solution set of (\ref{fdccon}) nonempty.

\begin{lemma}\label{lem41} Let $f_{dc}^*$ be defined in (\ref{fdccon}). Assume that ${\cal N}(A) \cap {\cal X} \neq\{{\bf 0}\}$. Then,
$f_{dc}^*<+\infty$ and the solution set of (\ref{fdccon}) is nonempty.  \end{lemma}

\begin{proof} First, the feasible set of (\ref{fdccon}) is nonempty due to $\mathcal{N}(A) \cap \mathcal{X} \neq\{\bf 0\}$. The objective function value is lower bounded, i.e., $\frac{\|{\h d}\|_{1}}{\| {\h d}\|_2} \geq 1$.  Suppose that there exists a minimizing sequence $\left\{\tilde{{\h d}}^{k}\right\}$ of (\ref{fdccon}) that is unbounded with $\frac{\left\|\tilde{{\h d}}^{k}\right\|_{1}}{\left\|\tilde{{\h d}}^{k}\right\|_{2}} \rightarrow f_{d c}^{*}$ as $k \rightarrow+\infty$.
Thus, $f_{d c}^{*}<+\infty$. Consequently, by defining ${\h d}^{k}=\frac{\tilde{{\h d}}^{k}}{\|\tilde{{\h d}}^{k}\|_{2}}$,
the sequence $\left\{{\h d}^{k}\right\}$ is satisfying
${\h d}^{k} \in {\cal F}_{0}$ and $\lim _{k \rightarrow+\infty} \frac{\|{\h d}^k\|_{1}}{\|{\h d}^k\|_{2}}=f_{d c}^{*} .$
Since the sequence of $\left\{{\h d}^{k}\right\}$ is bounded,
it has one accumulation point ${\h d}^{*} \in {\cal F}_{0}$.
 Therefore, the solution value of (\ref{fdccon}) is attainable by ${\h d}^{*}$.
\end{proof}
The following proposition shows that the optimal value of (\ref{L1o2Con}) is upper bounded by that of (\ref{fdccon}) when the feasible sets of both (\ref{L1o2Con}) and (\ref{fdccon}) are nonempty.

\begin{proposition} \label{prop42} Suppose that $\mathcal{N}(A) \cap \mathcal{X} \neq\{\bf 0\}$. Then, $f_{pc}^{*} \leq f_{dc}^{*}$. \end{proposition}

\begin{proof}
Since ${\cal N}(A)\cap{\cal X}\neq\{\bf 0\}$,  it leads to ${\cal F}_0\neq\phi$.
For any ${\hat {\h x}}\in{\cal H}$ and any ${\hat {\h d}}\in{\cal F}_0$,
we have ${\hat{\h x}}+\tau {\hat{\h d}}\in{\cal H}$ where $\tau>0$.
Thus, it leads to $f_{pc}^*=\inf_{{\h x}\in{\cal H}}\frac{\|\h x\|_1}{\|\h x\|_2}\le \frac{\|{\hat {\h x}}+\tau {\hat{\h d}} \|_1}{\| {\hat {\h x}}+\tau {\hat{\h d}}\|_2}$. Next, we have that $\lim_{\tau\to +\infty}\frac{\|{\hat{\h x}}+\tau {\hat{\h d}}\|_1}{\| {\hat{\h x}}+\tau {\hat{\h d}}\|_2}=\lim_{\tau\to +\infty}\frac{\|{\hat{\h x}}/\tau+{\hat{\h d}}\|_1}{\| {\hat{\h x}}/\tau+ {\hat{\h d}}\|_2}=\frac{\|{\hat{\h d}}\|_1}{\|{\hat{\h d}}\|_2}$.
Consequently, for any ${\hat {\h d}}\in{\cal F}_0$, it yields that $f_{pc}^*\le\frac{\|{\hat{\h d}}\|_1}{\|{\hat{\h d}}\|_2}$. By
taking infimum on both sides of the above inequality with respect to ${\hat{\h d}}$, we have
the desired inequality.
\end{proof}

\begin{theorem}\label{Theo4.3} Assume that ${\cal N}(A)\cap{\cal X} \neq \{\bf 0\}$.  Consider (\ref{L1o2Con}) and (\ref{fdccon}). Then,
$f_{pc}^*=f_{dc}^*$ if and only if there exists a minimizing sequence of (\ref{L1o2Con}) that is unbounded. \end{theorem}
\begin{proof} For the case of ${\cal X}={\mathbb R}^{n},$ it has been proved in \cite[Lemma 3.3]{ZengYuPong20} that $f_{p c}^{*}=f_{d c}^{*}$ if and only if there exists a minimizing sequence of (\ref{L1o2Con}) that is unbounded.
The proof for  the case of ${\cal X}=\mathbb{R}_{+}^{n}$ is similar to the case of ${\cal X}=\mathbb{R}^{n}$, and thus omitted here.
\end{proof}

\begin{corollary}\label{coro4.4} Suppose that $\mathcal{N}(A) \cap \mathcal{X} \neq\{\bf 0\}$. Consider (\ref{L1o2Con}) and (\ref{fdccon}).
If $f_{pc}^{*}<f_{d c}^{*}$, the solution set of (\ref{L1o2Con}) is nonempty. \end{corollary}

\begin{proof} It follows directly by combining Theorem \ref{Theo4.3}  and Proposition \ref{prop42}. \end{proof}

Next, we analyze the existence of the global solution of the penalized model (\ref{L1o2uncon}). In doing so,
we introduce a constrained problem parameterized by the vector $\h c$:
\begin{eqnarray}\label{paraL1o2}
\begin{array}{ll}
\min_{\h x}&\displaystyle{\frac{\|\h x\|_1}{\|\h x\|_2}}\\[0.2cm]
s.t.       &A{\h x}={\h c}, \;{\h x}\in {\cal X},\\
           \end{array}
\end{eqnarray}
and denote the optimal value of (\ref{paraL1o2}) as $f_{p c}^{*}(\h c)$ and $f^{*}_{p c}(\h c)<+\infty$ for any ${\h c} \in A {\cal X}$.

\begin{theorem} \label{Theo4.5} Suppose that $\mathcal{N}(A) \cap \mathcal{X} \neq\{\bf 0\}$. Consider (\ref{L1o2uncon}) and (\ref{paraL1o2}).  If $f_{pc}^*(\h c)< f_{dc}^*$ where
${\h c}={\text{Proj}}_{A{\cal X}}(\h b)$, the optimal value of (\ref{L1o2uncon}) can be attainable. \end{theorem}
\begin{proof} First, define $r({\h x})=\gamma\frac{\|\h x\|_1}{\|\h x\|_2}.$
Assume that $\left\{\mathbf{x}^{k}\right\}$ is a minimizing sequence of (\ref{L1o2uncon}), i.e.,
\begin{eqnarray}\label{Flim}  \lim _{k \rightarrow \infty} F\left({\h x}^k\right)=F^{*}. \end{eqnarray}
If the sequence $\left\{{\h x}^{k}\right\}$ is bounded, then it has a subsequence $\left\{{\h x}^{k_{j}}\right\}$ converging to some ${\h x}^{*}$. Hence, ${\h x}^{*}$ is an optimal solution of (\ref{L1o2uncon}). Otherwise, the sequence $\left\{{\h x}^{k}\right\}$ is unbounded, i.e., $\|\h x^k\|\to\infty$ as $k\to\infty$.
Since the sequence of $\{r(\h x^k)\}$ is bounded below and
$F^*$ is finite, it leads to the sequence of $\{\frac{1}{2}\|A{\h x}^k-{\h b}\|^2_2\}$ is bounded above.
It implies that the subsequence of $\{A{\h x}^k\}$ is bounded. Since $\{A{\h x}^k\}$ is bounded, it has a sequence
that converges to ${\h y}^*$. Without loss of generality, we assume
that $A{\h x}^k\rightarrow {\h y}^*$. Let ${\cal I}=\{j:\;\{x_j^k\}\;\mbox{is bounded}\}$.
Then, it follows that $\{A_{{\cal I}} {\h x}_{\cal I}^k\}$ is bounded.
This together with boundedness of $\{A{\h x}^k\}$ implies that  $\{A_{{\cal I}^c} {\h x}_{{\cal I}^c}^k\}$ is also bounded.
Next, for each $k$, we consider the following linear system,
$ A_{{\cal I}^c} {\h y}^k = A_{{\cal I}^c} {\h x}_{{\cal I}^c}^k,\;{\h y}^k\in{\cal Y},$
where ${\cal Y}:={\mathbb R}^{\sharp({\cal I}^c)}$ if ${\cal X}={\mathbb R}^n$ and ${\cal Y}:={\mathbb R}_+^{\sharp({\cal I}^c)}$ if ${\cal X}={\mathbb R}_+^n$.
Obviously, the solution set of the above linear system is nonempty due to at least one solution $\h y^k = {\h x}^k_{{\cal I}^c}$ for
each $k$. Using Hoffman's Error Bound \cite{Hoffman1952OnAS}, there exist a vector $\h y^k$ satisfying $A_{{\cal I}^c} {\h y}^k = A_{{\cal I}^c} {\h x}_{{\cal I}^c}^k$ and a constant
$\zeta>0$ depending only on $A_{{\cal I}^c}$ such that
\begin{eqnarray*} \|\h y^k\|\le \zeta \|A_{{\cal I}^c} {\h x}_{{\cal I}^c}^k\|.\end{eqnarray*}
By setting ${\hat {\h x}}^k=({\h x}_{{\cal I}}^k,{\h y}^k)$, it leads to
$A{\hat {\h x}^k} =A{\h x}^k$ and ${\hat{\h x}}^k\rightarrow {\h x}^*$ for convenience due to its boundedness of $\{{\hat{\h x}}^k\}$.
Obviously, ${\h y}^*=A{\h x}^*$ thanks to $A{\hat{\h x}}^{k_j} = A{\h x}^{k_j}$. In
the following, we divide into two cases to verify.\\
Case 1. If there exists two subsequences $\{{\h x}^{k_j}\}$ and
$\{{\hat{\h x}^{k_j}}\}$ such that
\begin{eqnarray} \label{Theo4.5:Ca1} \frac{\|{\h x}^{k_j}\|_1}{\| {\h x}^{k_j}\|_2}\ge \frac{\|{\hat{\h x}}^{k_j}\|_1}{\|{\hat{\h x}}^{k_j}\|_2},\;\;\forall \;j. \end{eqnarray}
Since $A{\hat {\h x}}^{k_j} = A{\h x}^{k_j}$, we have that
$F({\h x}^{k_j}) \ge  F({\hat{\h x}}^{k_j}).$
By using ${\hat{\h x}}^{k_j} \rightarrow {\h x}^*$ and $F$ lower semi-continuous, it yields that
$\mathop{\underline{\lim}}\limits_{j\to\infty} F({\hat{\h x}}^{k_j})\ge F({\h x}^*).$
On the other hand,
\begin{eqnarray*}  F^*=\mathop{{\lim}}\limits_{k\to\infty}F({{\h x}}^{k}) = \mathop{{\lim}}\limits_{j\to\infty}F({{\h x}}^{k_j})\ge \mathop{\underline{\lim}}\limits_{j\to\infty} F({\hat{\h x}}^{k_j})
\ge F({\h x}^*).\end{eqnarray*}
Invoking the definition of $F^*$,  it leads to $F({\h x}^*)=F^*$ and ${\h x}^*$ is an optimal solution.\\
Case 2. If there does not exist such two sequences $\{{\h x}^{k_j}\}$ and $\{{\hat{\h x}}^{k_j}\}$ satisfying (\ref{Theo4.5:Ca1}),
it implies there exists an index $K$ such that
$ \frac{\|{\h x}^k\|_1}{\|{\h x}^k\|_2}<\frac{\|{\hat{\h x}}^k\|_1}{\|{\hat{\h x}}^k\|_2},\;\;\forall\; k\ge K. $
Next, we further divide into two cases to verify.

\begin{itemize}
\item[(a)] Suppose that ${\h y}^*\in {\text{Proj}}_{(A{\cal X})}({\h b})$.
Then, the solution set of (\ref{paraL1o2}) with ${\h c}={\h y}^*$ is nonempty due to $f_{pc}^*({\h y}^*)< f_{dc}^*$.
We assume that ${\hat{\h x}}$ is an optimal solution of (\ref{paraL1o2}) with ${\h c}={\h y}^*$.
Since (\ref{Flim}) and $A{\h x}^k\to {\h y}^*$, then
\begin{eqnarray*}\lim_{k\to\infty} \frac{\|{\h x}^k\|_1}{\|{\h x}^k\|_2}=(F^*-\frac{1}{2}\|{\h y}^*-{\h b}\|_2^2)<+\infty.\end{eqnarray*}
Next, we verify that
\begin{eqnarray}\label{Theo4.5:Ca2a1}\lim_{k\to\infty} \frac{\|{\h x}^k\|_1}{\|{\h x}^k\|_2}\le \frac{\|{\hat{\h x}}\|_1}{\|{\hat {\h x}}\|_2} =f_{pc}^*({\h y}^*). \end{eqnarray}
Suppose not, i.e., $\lim_{k\to\infty} \frac{\|{\h x}^k\|_1}{\|{\h x}^k\|_2}> \frac{\|{\hat{\h x}}\|_1}{\|{\hat {\h x}}\|_2}$, it implies that $F({\hat{\h x}})<F^*$ since $A{\hat{\h x}}=A{\h x}^*={\h y}^*$. It contradicts  the definition of $F^*$.
Thus, (\ref{Theo4.5:Ca2a1}) holds.
We define ${\tilde{\h x}^k}:=\frac{{\h x}^k}{\|\h x^k\|_2}$. Taking  $k\to \infty$,
$ A{\h x}^k \rightarrow {\h y}^* \;\Rightarrow\; A{\tilde {\h x}}^k \rightarrow{\bf 0},$
 since $\|{\h x}^k\|_2\to \infty$.
Since ${\tilde {\h x}}^k$ is bounded, it has  a subsequence ${\tilde {\h x}}^{k_j}\rightarrow {\bar{\h x}}$ where ${\bar{\h x}}$ satisfies  ${\bar{\h x}}\in{\cal F}_0$.
\begin{eqnarray}\label{Theo4.5:Ca2a2}\mathop{\lim}\limits_{k\to\infty}\frac{\|{\h x}^k\|_1}{\|{\h x}^k\|_2}=\frac{\|{\bar{\h x}}\|_1}{\|{\bar{\h x}}\|_2}\ge f_{dc}^*>f_{pc}^*({\h y}^*), \end{eqnarray}
where the first inequality is due to ${\bar{\h x}}\in{\cal F}_0$ and the last inequality follows
from ${\h y}^*\in {\text{Proj}}_{A{\cal X}}(\h b)$. Note that the above inequality contradicts  (\ref{Theo4.5:Ca2a1}).
Hence, this case cannot happen.

\item[(b)] Suppose that ${\h y}^*\not\in{\text{Proj}}_{A{\cal X}}(\h b)$. Then,
we choose one vector ${\hat {\h y}}\in {\text{Proj}}_{A{\cal X}}({\h b})$.
Consider the constrained problem (\ref{paraL1o2})  with ${\h c}:={\hat{\h y}}$.
Since $f^*_{pc}({\hat{\h y}})<f_{dc}^*$,
the solution set of (\ref{paraL1o2}) with ${\h c}:={\hat{\h y}}$
is nonempty due to Corollary \ref{coro4.4}. We assume that $\breve{{\h x}}$ is an optimal solution of (\ref{paraL1o2}) with ${\h c}:={\hat{\h y}}$.
    Similar to case (a), one can derive a version of (\ref{Theo4.5:Ca2a2}) as follows
      \begin{eqnarray} \label{Theo4.5:Ca2a3} \mathop{\lim}\limits_{k\to\infty} \gamma\frac{\|\h x^k\|_1}{\|\h x^k\|_2}=\gamma\frac{\|{\bar{\h x}}\|_1}{\|{\bar{\h x}}\|_2}\ge \gamma f_{dc}^*, \end{eqnarray}
      where the definition of  ${\bar{\h x}}$ is the same as Case (a).
      By noting ${\hat{\h y}}\in{\text{Proj}}_{A{\cal X}}(\h b)$ and
      ${\h y}^*\in A{\cal X}$ since $A{\h x}^k\in A{\cal X}$ and $A{\h x}^k \to {\h y}^*$,
      it leads to $\frac{1}{2}\|{\hat{\h y}}-{\h b}\|^2\le \frac{1}{2}\|{\h y}^*-{\h b}\|^2.$
      Thus, $\frac{1}{2}\|A\breve{\h x}-{\h b}\|^2\le \frac{1}{2}\|{\h y}^*-{\h b}\|^2$.
      By noting $r(\breve{\h x})=\gamma f_{pc}^*({\hat{\h y}})<\gamma f_{dc}^*$,
      and combining with the above inequality, it yields that
     \begin{eqnarray} \label{xxtmp} F(\breve{\h x})<\gamma f_{dc}^*+\frac{1}{2}\|{\h y}^*-{\h b}\|^2.\end{eqnarray}
      On the other hand, combining  (\ref{Theo4.5:Ca2a3}) with the fact of $A{\h x}^k \rightarrow {\h y}^*$, we have
      \begin{eqnarray*} &&\gamma f_{dc}^* +\frac{1}{2}\|{\h y}^*-{\h b}\|^2\le\mathop{\lim}\limits_{k\to\infty}\left( \gamma\frac{\|\h x^k\|_1}{\|\h x^k\|_2}+\frac{1}{2}\|A{\h x}^k-{\h b}\|_2^2\right)=F^*.\end{eqnarray*}
     \noindent In view of (\ref{xxtmp}) and the above inequality,  it leads to $F(\breve{\h x})<F^*$ which
      contradicts the definition of $F^*$.
\end{itemize}
Thus, the sequence $\{{\h x}^k\}$ is bounded, and thus it has an accumulation point ${\h x}^*$ which
is an optimal solution of (\ref{L1o2uncon}).
\end{proof}
Next, we present the theorem on  the existence of global solutions of (\ref{L1o2Con}) and (\ref{L1o2uncon}).
The proof follows the line of arguments as in \cite[Theorem 3.4]{ZengYuPong20}, thus  omitted  here.

\begin{theorem} \label{theo4.6}
Consider (\ref{L1o2Con}) and (\ref{L1o2uncon}). Suppose that ${\cal N}(A)$ has the $\mu$-spherical section property for some
$\mu>0$. Then,  the following assertions hold:
\begin{itemize}
\item[(i)] If there exists ${\hat{\h x}}\in{\mathbb R}^n$ such that $\|\hat{\h x}\|_0<m/\mu$, ${\hat{\h x}}\in {\cal X}$ and $A{\hat{\h x}}={\h b}$, then
the set of optimal solution of (\ref{L1o2Con}) is nonempty.
\item[(ii)] If there exists ${\hat{\h x}}\in\mathbb R^n$ such that $\|\hat{\h x}\|_0<m/\mu$, ${\hat{\h x}}\in{\cal X}$, and $A{\hat{\h x}}={\h c}$
where ${\h c}={\text{Proj}}_{A{\cal X}} (\h b)$, then the set of optimal solutions of (\ref{L1o2uncon}) is nonempty.
\end{itemize}  \end{theorem}

Next, we consider   how to guarantee ${\bf 0}$  not being  a globally optimal solutions for (\ref{L1o2Con}) and (\ref{L1o2uncon}). In view of ${\h b}\neq {\bf 0}$, ${\bf 0}$ cannot be a globally optimal solution of the constrained model (\ref{L1o2Con}).
In Theorem \ref{theo4.7}, we provide sufficient conditions to guarantee that ${\bf 0}$ cannot be a globally optimal solution of (\ref{L1o2uncon}).

\begin{theorem}
\label{theo4.7}
Suppose that one of the following assumptions holds:
\begin{itemize}
\item[(i)] ${\h b}\in A{\cal X}$ and $0<\gamma<\frac{\|{\h b}\|_2^2}{2(\sqrt{n}-1)}$;
\item[(ii)] There exists a vector ${\hat {\h x}}\in{\cal X}$ such that $\|A{\hat{\h x}}-{\h b}\|_2\le\varepsilon\;(\varepsilon \ll \|{\h b}\|_2)$
and $0<\gamma<\frac{\|{\h b}\|_2^2 - \varepsilon^2}{2(\sqrt{n}-1)}$.
\end{itemize}
Then, the optimal solution of (\ref{L1o2uncon}) cannot be ${\bf 0}$.
\end{theorem}

\begin{proof} (i) We use contradiction to show it. Suppose that ${\bf 0}$ is a global solution of (\ref{L1o2uncon}).
 Since ${\h b}\in A{\cal X}$, we choose a vector ${\tilde {\h x}}\in{\cal X}$ such that $A{\tilde {\h x}}={\h b}$.
Then, it leads to $\gamma\sqrt{n}<\gamma +\frac{1}{2}\|\h b\|^2$ which
implies that $F({\tilde {\h x}})<F({\bf 0})$. It contradicts to ${\bf 0}$ being a global
solution of (\ref{L1o2uncon}). (ii) The proof is similar to (i), thus omitted here.
\end{proof}
\begin{remark}
The assumptions of (i) and (ii) in Theorem \ref{theo4.7} correspond to
the cases of noiseless and noisy observations, respectively.
\end{remark}

\section{Sparse property}\label{RecoveryT}

We  demonstrate  the sparsity of the local minimizers of (\ref{L1o2Con}) and (\ref{L1o2uncon})
in the sense
that minimizing $L_1/L_2$ or $(L_1/L_2)^+$ only extract linearly independent columns from the
sensing matrix $A$. With this, we provide a much more easily checkable exact recovery condition than \cite[Theorem III.2]{YEX14}  for the constrained model
both for arbitrary and nonnegative signals.
\begin{theorem} \label{theo4.9} Let ${\h x}^*$ $({\h x}^*\neq{\bf 0})$ be a local minimizer of the constrained problem (\ref{L1o2Con})
and $\Lambda^*={\text{supp}}(\h x^*)$. Then,
$(A_{{\Lambda}^*})^\top  (A_{\Lambda^*}) \succ {\bf 0}.$
\end{theorem}

\begin{proof}
We divide into two cases to verify.\\
Case 1. ${\cal X}={\mathbb R}^n$. Let ${\h x}^*$ be a local minimizer of the constrained model (\ref{L1o2Con}).
We use contradiction. Suppose the columns of $A_{\Lambda^*}$ are linearly dependent; then
there exists ${\bf v}\neq{\bf 0}$ and ${\h v}\in{\cal N}(A)$ such that ${\text{supp}}(\h v)\subseteq \Lambda^*$.
For any fixed neighborhood ${\cal B}_r(\h x^*)$ of ${\h x}^*$, we scale $\h v$ so that
$ \|\h v\|_2<\min\{\mathop{\min}\limits_{i\in{{\Lambda}^*}}|x_i^*|,\;r\}. $
Consider two feasible vectors in ${\cal B}_r({\h x}^*)$, ${\hat{\h x}}={\h x}^*+{\h v}$ and $\check{\h x}={\h x}^*-{\h v}$.
Since ${\text{supp}}(\h v)\subseteq \Lambda^*$, we have ${\text{supp}}({\hat {\h x}})\subseteq {\Lambda}^*$ and
${\text{supp}}(\check{\h x})\subseteq {\Lambda}^*$. Since for any $i\in{\Lambda}^*$,
$|x_i^*|\pm {\text{sign}}(x_i^*) v_i\ge \mathop{\min}\limits_{i\in{\Lambda}^*}|x_i^*|-\|\h v\|_2>0,\;\forall\; i\in{\Lambda}^*$.
Thus,  $({\h x}^*\pm{\h v})_i={\text{sign}}(x_i^*) (|x_i^*|\pm {\text{sign}}(x_i^*) v_i)$ for any  $i\in\Lambda^*.$
It implies that ${\h x}^*$, ${\hat{\h x}}$ and $\check{\h x}$ are located in the same octant.
Consequently,
$ \|\h x^*\|_1 = \frac{1}{2}\left( \| {\hat {\h x}}\|_1 + \| \check{\h x }\|_1 \right),\;\|\h x^*\|_2 <\frac{1}{2}\left( \| {\hat {\h x}}\|_2 + \| \check{\h x }\|_2 \right)$.
Suppose not. Then,
there exists a positive scalar $\kappa$($\neq 1$) such that $\check{\h x }=\kappa {\hat {\h x}}$ which contradicts
the facts of $A{\hat{\h x}}={\h b}$ and $A{\h {\check x}}={\h b}$.
Finally, it yields that
$\frac{\|\h x^*\|_1}{\|\h x^*\|_2} >\min\left\{ \frac{\|{\hat{\h x}}\|_1}{\|{\hat{\h x}}\|_2}, \frac{\|{\check{\h x}}\|_1}{\|{\check{\h x}}\|_2}\right\},$
which contradicts the fact that ${\h x}^*$ is a local minimizer.\\
Case 2. ${\cal X}={\mathbb R}_+^n$. The  proof is similar to Case 1, thus
 omitted  here.
\end{proof}

Next, we show that the conclusion of Theorem \ref{theo4.9} also holds for the unconstrained model.

\begin{theorem} \label{theor4.10}
Let ${\h x}^*$ be a local minimizer of the unconstrained problem (\ref{L1o2uncon}) and
${\Lambda^*}={\text{supp}}({\h x}^*)$. Then,
$  (A_{\Lambda^*})^\top (A_{\Lambda^*})\succ{\bf 0}.$ \end{theorem}

\begin{proof}
First, we show that ${\h x}^*$ is also a local minimizer of the constrained problem (\ref{L1o2Con}) where
${\h b}:=A{\h x}^*$. Suppose not.
Then, for any $r>0$, there exists ${\h x}_r\in {\cal B}_r(\h x^*)\cap{\cal X}$
such that $A{\h x}_r = A{\h x}^*$ and $\frac{\|{\h x}_r\|_1}{\|{\h x}_r\|_2}< \frac{\|{\h x}^*\|_1}{\|{\h x}^*\|_2}.$
It further implies that $F({\h x}_r) <F({\h x}^*)$ where $F$ is defined in (\ref{L1o2uncon}),
which contradicts  that ${\h x}^*$ is a local minimizer of (\ref{L1o2uncon}). Therefore, ${\h x}^*$  is a local
minimizer of (\ref{L1o2Con}) where
${\h b}:=A{\h x}^*$. By invoking Theorem \ref{theo4.9}, the conclusion follows
directly.
\end{proof}

\begin{remark}
From Theorems \ref{theo4.9} and \ref{theor4.10}, we see that if a computed solution ${\h x}$ from the model (\ref{L1o2Con}) or (\ref{L1o2uncon}) fails to extract linearly independent columns from the
sensing matrix $A$.
Then, ${\h x}$ cannot be a local
minimizer.
\end{remark}
The next lemma presents a sufficient and necessary condition for characterizing ${\h x}\in{\cal H}_{L}$ which turns out to be checkable.

\begin{lemma}\label{tmpadd} ${\h x}\in{\cal H}_{L}$ if and only if  $A^\top_{\Lambda} A_{\Lambda}\succ{\bf 0}\;\mbox{where}\;\Lambda=\supp(\h x)$.
\end{lemma}
\begin{proof}
We use contradiction to show  the direction of ``only if".
 Let $\alpha=\sharp(\Lambda)$. 
Suppose not. It implies that there exists a vector ${\h v}\in{\mathbb R}^\alpha\;({\h v}\neq {\bf 0})$ such
that ${\h v}\in{\cal N}(A_{\Lambda})$. Thus, there exists a vector ${\tilde {\h v}}\in{\mathbb R}^n$ such that $({\tilde {\h v}})|_{\Lambda}={\h v}$ and $({\tilde {\h v}})|_{\Lambda^c}={\bf 0}$.
 We define
                                      an index set: $\Lambda_{{\tilde {\h v}}}:=\supp({\tilde{\h v}})$.
                                      Note that $\supp({\tilde {\h v}})\subseteq \Lambda$. Let
                                      $\zeta:=\min_{i\in\Lambda_{{\tilde {\h v}}}} |\frac{x_i}{{\tilde v}_i} |>0$
                                      and ${\tilde i}\in\arg\min_{i\in\Lambda_{{\tilde {\h v}}}}  |\frac{x_i}{{\tilde v}_i} |$.
                                      Next, we define the vector:
                                      ${\tilde {\h x}} = {\h x}-{\text{sign}}(x_{\tilde i}{\tilde v}_{\tilde i})\zeta {\tilde {\h v}}.$
                                      Then, ${\tilde {\h x}}\in{\cal H}$, and it is much sparser than $\h x$ and $\supp({\tilde {\h x}})\subseteq
                                      \supp(\h x)$ which contradicts  ${\h x}\in {\cal H}_L$.
Second, for  the direction of ``if", we also use contradiction.
Suppose there exists a feasible solution $\h y(\neq{\h x})$ such that ${\supp(\h y)}\subseteq{\supp}({\h x})$, and $A{\h y}={\h b}$,
${\h y}\in{\cal X}$. Then, define ${\h v}:={\h y}-{\h x}(\neq{\bf 0})\in{\cal N}(A_{\Lambda})$ which contradicts  $(A_{\Lambda})^\top(A_{\Lambda})\succ {\bf 0}$.
\end{proof}

Combining Theorems  \ref{theo4.9}, \ref{theor4.10} and Lemma \ref{tmpadd}, we conclude the following corollary.

\begin{corollary}  \label{coro4.11}
We have these facts hold:
\begin{itemize}
\item[(i)] Suppose that $rank(A)=m$ and ${\h x}^*$ is a local minimizer of (\ref{L1o2Con}) and (\ref{L1o2uncon}), the sparsity of ${\h x}^*$ is at most $m$.
\item[(ii)] The model (\ref{L1o2Con}) has a finite number of local minimizers.
\item[(iii)] If ${\h x}^*$ is a local minimizer of (\ref{L1o2Con}), then ${\h x}^*\in{\cal H}_{L}$ where ${\cal H}_{L}$ is defined in
Definition \ref{def2.3}.
\end{itemize}
\end{corollary}
\begin{proof} The proof for (i) and (ii) are elementary, thus omitted.
(iii) It follows from Theorems \ref{theo4.9}, \ref{theor4.10} and Lemma \ref{tmpadd} directly.
\end{proof}
\noindent
Equipped with Lemma \ref{tmpadd}, we provide an  exact recovery condition of the constrained model (\ref{L1o2Con})
 which extends  \cite[Theorem III.2]{YEX14} to nonnegative/arbitrary signal and
turns out to be much  easier to check.
The proof is analogous to \cite[Theorem III.2]{YEX14}, thus omitted.

\begin{theorem} \label{theo5.6} If ${\h x}_0$ uniquely solves (\ref{CSzero}) and $\|{\h x}_0\|_0=s$ and
if
\begin{eqnarray}\label{tempexR} \kappa(\h x)>\frac{(\sqrt{\|\h x\|_0}-\sqrt{\|\h x\|_0-s})^2}{s},\;\;\forall\; {\h x}\in{\tilde{\cal F}}\backslash\{{\h x}_0\},\end{eqnarray}
$ {\tilde {\cal F}}:=\{{\h x}\in{\cal H}\;|\;\Lambda=\supp (\h x),\;A_{\Lambda}^\top A_{\Lambda} \succ {\bf 0} \},$
where ${\cal H}$ is defined in (\ref{L1o2Con}).
Then, ${\h x}_0$ also uniquely solves (\ref{L1o2Con}).
\end{theorem}
\section{Computational approach}\label{ADMMS}
We focus on solving (\ref{L1o2uncon}) with
${\cal X}=\mathbb R_+^n$. Inspired by \cite{Tao20}, we derive the closed-form solution of the proximal operator of $(L_1/L_2)^+$, and accompanied by a practical solver for finding one global solution.
\subsection{Proximal operator}  Define a proximal operator of $\left(L_{1} / L_{2}\right)^{+} (:=\frac{\|\h x\|_1}{\|\h x\|_2}+\iota_{{\mathbb R}^n_+}(\h x))$ with a parameter $\rho>0$ as
\begin{eqnarray}\label{close} \prox_{[({L}_1/{L}_2)^+/\rho]}(\h q):={\arg\min}_{\h x\in{\mathbb R}_+^n}\Big(\frac{\|\h x\|_1}{\|\h x\|_2}+\frac{\rho}{2}\|\h x-\h q\|_2^2\Big). \end{eqnarray}
It follows from \cite[Definition 1.23]{RockWets} and \cite[Theorem 1.25]{RockWets}, the solution set of (\ref{close}) is nonempty.
Obviously, if ${\h q} \in{\mathbb R}^n_{-}$, the solution of (\ref{close}) is ${\bf 0}$. Next,  Example 1 shows that the optimal solution of (\ref{close}) may not be unique.

\noindent{\bf Example 1.} Let $n=2$ and $q_{1}=q_{2}=\sqrt{2(\sqrt{2}-1)}$. Consider an objective function
\begin{eqnarray*} \min _{{\h x} \in{\mathbb R}^n_+} \frac{\left|x_{1}\right|+\left|x_{2}\right|}{\sqrt{x_{1}^{2}+x_{2}^{2}}}+\frac{1}{2}\left(x_{1}-q_{1}\right)^{2}+\frac{1}{2}\left(x_{2}-q_{2}\right)^{2}.
\end{eqnarray*}

\noindent Indeed, it has three globally optimal solutions:
\begin{eqnarray*}{\h x}_{1}=(\sqrt{2(\sqrt{2}-1)}, 0)^{\top},\; {\h x}_2=(0, \sqrt{2(\sqrt{2}-1)})^{\top}\;\;\mbox{and}\;\;{\h x}_3=\left(\sqrt{2(\sqrt{2}-1)}, \sqrt{2(\sqrt{2}-1)}\right)^{\top}.\nn\\
\label{x123}\end{eqnarray*}

Next, we characterize one of globally optimal solutions of  $\operatorname{Prox}_{[({L}_1/{L}_2)^+/\rho]}$ $({\h q})$
in a closed-form.

\begin{theorem}\label{Theo21} Given $\h q\in{\mathbb R}^n$ and ${\h q} \not\le \h 0$ and $\rho>0.$ 	
We can sort ${\h q}$ in a descending order in a way of ${q}_{\pi(1)}\ge \cdots\ge { q}_{\pi(\nu)}>0\ge q_{\pi(\nu+1)}\ge\cdots\ge{q}_{\pi(n)}$
where $\pi$ is a proper permutation of $[n]$.
Then, the following assertions hold:
	\begin{enumerate}
		\item[(i)] There exists an optimal solution ${\bar{\h x}}$ of \eqref{close} such that it has the same descending
order  as ${\h q}$, i.e.,
		$$ \bar{ x}_{\pi(1)} \ge\cdots\ge {\bar x}_{\pi(\nu)}\ge 0= {\bar x}_{\pi(\nu+1)}=\cdots={\bar x}_{\pi(n)}.$$
		\item[(ii)] We denote the multiplicity of the largest magnitude in $\h q$ as $\mu$,  i.e., $q_{\pi(1)}=\cdots=q_{\pi(\mu)}> q_{\pi(\mu+1)}$.

One of the following assertions holds:
 \begin{itemize}
  \item[(a)] If $0<\rho\le 1/(q_{\pi(1)}^2)$, then (\ref{close}) has a one-sparse solution given by
			\begin{eqnarray*}\label{eq:prox4x_case2}
		{\bar{ x}}_{\pi(i)}=\left\{\begin{array}{ll}
		{ q}_{\pi(i)} &i=1;\\
		0 &\mbox{otherwise}.
		\end{array}
		\right.
		\end{eqnarray*}
  \item[(b)]  If $\rho>\displaystyle{\frac{1}{q_{\pi(1)}^2}}$, there exist an integer $t$ ($t\le \nu$) and a scalar pair of $(a,r)$ such that $(Q^t=\sum_{i=1}^t q_{\pi(i)})$
		\begin{subequations} \nn
			\begin{numcases}{\hbox{\quad}}
			\nn \frac{a^2}{r^3}-\rho a+\rho Q^t-\frac{t}{ r}=0,\\[0.0cm]
			\nn r^3-\left(\sum_{i=1}^{t} { q}_{\pi(i)}^2\right) r+\frac { Q^t-a}\rho=0,
			\end{numcases}
		\end{subequations}
and the $r$ is also satisfied with
\begin{eqnarray*}\label{condadj}
q_{\pi(t)}-\frac 1 {\rho r}>0 \ \mbox{  and  }\ q_{\pi(t+1)}-\frac{1}{\rho r}\le 0,
 \end{eqnarray*}
 and the vector  ${\bar{\h x}}$ is characterized by
		\begin{eqnarray*}\label{eq:prox4x}
		{\bar{ x}}_{\pi(i)}=\left\{\begin{array}{ll}
		\displaystyle{\frac{\rho { q}_{\pi(i)} -\frac{1}{r}}{\rho-\frac{a}{r^3}}}&1\le i\le t,\\
		0 &\mbox{otherwise},
		\end{array}
		\right.
		\end{eqnarray*}
is an optimal solution of (\ref{close}), where $t=\|{\bar{\h x}}\|_0$, $a=\|{\bar{\h x}}\|_1$, $r=\|{\bar{\h x}}\|_2$.
    \end{itemize}

	\end{enumerate}
\end{theorem}
\begin{proof}
(i) First, we verify that for any global solution ${\bar{\h x}}$, we have
\begin{eqnarray} \label{key1} q_i >q_j \;\Rightarrow \; {\bar x}_i\ge {\bar x}_j.\end{eqnarray}
We use contradiction. Define the objective function of (\ref{close}) by $f({\h x})=\frac{\|{\h x} \|_1}{\|{\h x}\|_2}+\iota_{\ge 0}(\h x)+ \frac{\rho}{2}\|\h x -{\h q}\|^2_2$.
Suppose not. It means that ${\bar x}_i <{\bar x}_j$. Then,
we exchange these two entries in ${\h {\bar x}}$ to obtain
a new vector ${\hat{\h x}}$. Then,
$f({\hat{\h x}})<f({\bar{\h x}})$ since
$(q_i-{\bar{x}_i})^2+(q_j-{\bar{x}_j})^2> (q_i-{\bar{x}_j})^2+(q_j-{\bar{x}_i})^2.$
It contradicts  ${\bar{\h x}}$ being a global solution. Thus, (\ref{key1}) holds.
 Furthermore, if ${\bar x}_i$, ${\bar x}_j>0$, we can strengthen the conclusion  in (\ref{key1}) to
 ``${\bar x}_i> {\bar x}_j$".
Indeed, by invoking the optimality conditions of (\ref{close}), it leads to
 \begin{eqnarray}\label{addopt}({\h x}-{\bar{\h x}})^\top\left(\frac{\bf 1}{\|{\bar{\h x}}\|_2}-\frac{\|{\bar {\h x}}\|_1}{\|{\bar {\h x}}\|_2^3}{\bar{\h x}}+\rho({\bar{\h x}}-{\h q}) \right) \ge {\bf 0},\;\;\;\;\forall{\h x}\in {\mathbb R}_+^n.  \end{eqnarray}
 Define $\Upsilon :=\displaystyle{ \frac{\bf 1}{\|{\bar{\h x}}\|_2}-\frac{\|{\bar{\h x}}\|_1}{\|{\bar{\h x}}\|_2^3}{\bar{\h x}} +\rho({\bar{\h x}}-{\h q})}$. Since ${\bar x}_i$, ${\bar x}_j>0$,
 it leads to
 \begin{eqnarray}\label{tm1}\Upsilon_i=\Upsilon_j=0.\end{eqnarray}
 Suppose not. Then, ${\bar x}_i= {\bar x}_j$. It follows from the above equality that $q_i=q_j$
 which contradicts $q_i>q_j$. Therefore, ${\bar x}_i> {\bar x}_j$.
 If there exists several entries of ${\h q}$ with the same value, the corresponding entries in ${\bar{\h x}}$ can be
arranged in a descending order.
 Thus, there exists a global solution ${\bar{\h x}}$ such that
\begin{eqnarray} \label{des} {\bar x}_{\pi(1)}\ge {\bar x}_{\pi(2)}\ge\cdots\ge {\bar x}_{\pi(n)}.  \end{eqnarray}
Next, we show that
\begin{eqnarray}\label{Theorem3.2:prot1} q_i \le 0 \Rightarrow {\bar x}_i = 0.\end{eqnarray}
We use contradiction. Suppose not. Then, there exists at least one index ${\hat i}$ such that
 $q_{\hat i}\le 0$ and ${\bar x}_{\hat i}>0$. 
 It follows from (\ref{addopt}) that
 \begin{eqnarray} \label{Theorem3.2:prot2} {\bar x}_{\hat i}>0 \Rightarrow \Upsilon_{\hat i} = 0.\end{eqnarray}
 In the following, we divide into two cases to verify.\\
 Case 1. $\rho\ge a/r^3$. Since ${\bar x}_{\hat i}>0$,
$\Upsilon_{\hat i}= \frac{1}{\|{\bar {\h x}}\|_2}+(\rho-\frac{a}{r^3}){\bar x}_{\hat i}-\rho q_{\hat i}>0$ due to $q_{\hat i}\le 0$ and $\rho-a/r^3\ge0$, and it
 contradicts  (\ref{Theorem3.2:prot2}).\\
 Case 2. $\rho <a/r^3$. Invoking (\ref{Theorem3.2:prot2}), it leads to
 \begin{eqnarray} \label{case2} {\bar x}_i=\displaystyle{\frac{\rho q_i - \frac{1}{r}}{\rho - \frac{a}{r^3}},\;\;\;\;\;\;\forall\; i\in{\text{supp}}({\bar{\h x}})}. \end{eqnarray}
 We define two index sets: $\Lambda^+:=\{i\;|\; q_i>0\}$ and $\Lambda^-:=\{i\;|\;q_i\le 0\}$.
 We divide ${\bar{\h x}}$ into two parts: ${\bar {\h x}}^+ ={\bar {\h x}}\mid_{\Lambda^+}$ and ${\bar {\h x}}^- ={\bar {\h x}}\mid_{\Lambda^-}$. By assumption, we know that
 ${\bar{\h x}}^-\neq {\bf 0}$. Thus, ${\bar{\h x}}^+\neq{\bf 0}$ due to (\ref{des}). Picking up $i
 \in {\Lambda^+}$ and setting $j={\hat i}\in \Lambda^-$ (i.e., ${\bar x}_j>0$).
 Since $q_i>q_j$ and ${\bar x}_j>0$, it leads to ${\bar x}_i>0$ due to (\ref{key1}).
 It implies that $i,j\in{\text{supp}}({\bar{\h x}})$.
 Consequently, it follows from (\ref{case2}) that
 \begin{eqnarray}\label{Theorem3.2:prot4} {\bar x}_i<{\bar x}_j,\end{eqnarray}
 due to $\rho<a/r^3$. On the other hand, since ${\bar{\h x}}$ is an optimal solution, we have
 proved that
 \begin{eqnarray*} q_i>q_j \Rightarrow {\bar x}_i> {\bar x}_j,\end{eqnarray*}
 which contradicts  (\ref{Theorem3.2:prot4}). Therefore, the assertion (\ref{Theorem3.2:prot1}) holds. Thus, the assertion (i) follows immediately.
 (ii) In view of (\ref{Theorem3.2:prot1}), it implies that the minimization problem (\ref{close}) amounts to solving a low-dimension
 minimization problem:
 \begin{eqnarray}\label{equpro} {\arg\min}_{{\h y}\in \mathbb{R}^{\nu}}\left( \frac{\|\h y\|_1}{\|\h y\|_2}+\frac{\rho}{2}\|{\h y}-{\h p}\|_2^2\right), \end{eqnarray}
 where ${\h p}={\h q}\mid_{\sigma}$ and $\sigma:=\{\pi(1),\cdots,\pi(\nu)\}$. Denoting ${\bar{\h y}}$ as an optimal solution of (\ref{equpro}), the vector ${\h x}$ defined by ${\h x}|_{\sigma}={\h {\bar y}}$ and ${\h x}|_{{\sigma}^c}={\bf 0}$
                                        is an optimal solution of (\ref{close}). By invoking \cite[Theorem 3.3]{Tao20}, the assertion of (ii) follows immediately.
\end{proof}

 For finding one global  solution of (\ref{equpro}),
a fast solver has been developed in \cite[Algorithm 3.1]{Tao20}.
 Indeed, Algorithm 3.1 of \cite{Tao20} aims to find one  global solution of $\prox_{[({L}_1/{L}_2)^+/\rho]}(\h q):={\arg\min}_{\h x\in{\mathbb R}^n}\Big(\frac{\|\h x\|_1}{\|\h x\|_2}+\frac{\rho}{2}\|\h x-\h q\|_2^2\Big)$. In \cite[Theorem 3.3]{Tao20}, one global solution  of
 $\prox_{[({L}_1/{L}_2)^+/\rho]}(\h q)$ has been characterized in a closed-form.
 It includes two cases: (a)  If $0<\rho\le 1/(q_{\pi(1)}^2)$, there is a one-sparse solution;
 (b) If $\rho>1/q_{\pi(1)}^2$, there is a $t$-sparse solution.
 Algorithm 3.1 \cite{Tao20} either returns a one-sparse solution for case (a) or produces a $t$-sparse solution for case (b). For the latter case, it adopts a bisection search to find the true sparsity $t$ and incorporates the fixed-point iterative method
to get the unique solution pair  $(a,r)$  \cite[Lemma 3.6]{Tao20}  of the two-dimension nonlinear system. With this $(a,r)$, Algorithm 3.1 of \cite{Tao20} computes
the  $t$-sparse solution in a closed-form.
More  discussions can be found in \cite[Section 3]{Tao20}.
In summary, an overall algorithm for finding ${\bar{\h x}}\in\prox_{[({L}_1/{L}_2)^+/\rho]}(\h q) $ with ${\h q}\in{\mathbb R}^n$ and $\rho>0$ is presented in Algorithm \ref{subx:admm:Bes}.

\begin{algorithm}[t]
	\caption{Finding a solution of $\prox_{[({L}_1/{L}_2)^+/\rho]}(\h q)$}
	\label{subx:admm:Bes}
	\begin{algorithmic}[1]
		\Require{$\rho >0$, ${\h q}\in{\mathbb R}^n$, ${q}_{\pi(1)}\ge \cdots\ge { q}_{\pi(\nu)}>0\ge q_{\pi(\nu+1)}\ge\cdots\ge{q}_{\pi(n)}$. Set ${\h p}={\h q}|_{\sigma}$ where $\sigma=\{\pi(1),\cdots,\pi(\nu)\}$.}
   \State{Using  \cite[Algorithm 3.1]{Tao20} to find ${\h {\bar y}}\in \prox_{[({L}_1/{L}_2)/\rho]}(\h p)$.}
   \State{Define ${\bar {\h x}}|_{\sigma}={\h {\bar y}}$ and ${\bar {\h x}}|_{{\sigma}^c}={\bf 0}$.}
	\State{ {\bf Output} ${\bar {\h x}}$.}
	\end{algorithmic}
\end{algorithm}
\subsection{ADMM for solving  (\ref{L1o2uncon}) with ${\cal X}=\mathbb R_+^n$}
Although there exist a few different ways for reformulating the unconstrained model (\ref{L1o2uncon}) with $\mathcal{X}=\mathbb{R}_+^{n}$,
most of them  result in a scheme of ADMM with violation of convergence guarantee \cite{LiPong15,HongLuoRazaviyayn16,WangYinZeng15}.
 Equipped with the newly-derived
solution of the proximity of $(L_1/L_2)^+$, we apply ADMM to  (\ref{L1o2uncon}) in a particular splitting way:
\begin{eqnarray}\label{ADMMI}
\begin{array}{ll}
\min_{\h x,\h y\in{\cal X}}&\displaystyle{\gamma\frac{\|\h x\|_1}{\|\h x\|_2}}+\frac{1}{2}\|A{\h y}-{\h b}\|^2_2\\[0.2cm]
s.t.                    &\h x=\h y,\;{\h x}\in {\cal X}.
\end{array}
\end{eqnarray}
The  augmented Lagrangian of (\ref{ADMMI}) is defined by
\begin{eqnarray}\label{eq:augL}
{\cal L}_{{\cal A}}(\h x,\h y,\h z)\!=\!\gamma\frac{\|\h x\|_1}{\|\h x\|_2}\!+\!\iota_{\cal X}(\h x)\!+\!\frac{1}{2}\|A{\h y}-{\h b}\|^2
+\h z^\top(\h x-\h y)\!+\!\frac{\beta}{2}\|\h x-\h y\|_2^2,
\end{eqnarray}
where $\h z$ is the Lagrangian multiplier and $\beta>0$ is the penalty parameter.
Given $({\h y}^k,{\h z}^k)$, the ADMM scheme generates the iterative sequence $\{\h w^k\}$  (${\h w}^k=({\h x}^k,{\h y}^k,
{\h z}^k)$) as follows,
\begin{subequations} \label{ADMMschI}
	\begin{numcases}{\hbox{\quad}}
	\label{xsub-ADMMI}\h x^{k+1}\in{\arg\min}_{\h x\in{\cal X}} {\cal L}_{{\cal A}}(\h x,\h y^k,\h z^k),\\[0.0cm]
	\label{ysub-ADMMI}\h y^{k+1}={\arg\min}_{\h y} {\cal L}_{{\cal A}}(\h x^{k+1},\h y,\h z^k),\\[0.0cm]
	\label{zsub-ADMMI}\h z^{k+1}=\h z^k+\beta(\h x^{k+1}-\h y^{k+1}).
	\end{numcases}
\end{subequations}
The $\h x$-subproblem (\ref{xsub-ADMMI})  amounts to
$\h x^{k+1}\in\prox_{[\frac{\gamma}{\beta}({L}_1/{L}_2)^+]}({\h y}^k-{\h z}^k/\beta).$
By using  Sherman-Morrison-Woodbury Theorem,  the $\h y$-subproblem \eqref{ysub-ADMMI} can be given by a more efficient
scheme:
\begin{eqnarray}\label{ynew}\h y^{k+1}\!=\!M \left( \frac{A^\top{\h b}}{\beta}+\frac{\h z^k}{\beta}+\h x^{k+1}\right),\end{eqnarray}
where $M\!=\!I_n-\frac{1}{\beta}A^\top (I_m+\frac{1}{\beta}AA^\top)^{-1}A$ since $m\ll n$.
We summarize the overall  scheme in Algorithm \ref{alg:ADMMschI}, and denote it by ADMM$_p^+$.

\begin{algorithm}[t]
	\caption{ADMM$_p^+$}
	\label{alg:ADMMschI}
	\begin{algorithmic}[1]
		\Require{$A\in{\mathbb R}^{m\times n}$, $\h b\in{\mathbb R}^m$, $\beta, \;\varepsilon>0$. }
		\State Initialize: $\h y^0=\h z^0$.
		\While{$k<k_{\max}$ or $\|\h x^{k-1}-\h x^{k}\|/\|\h x^{k}\|>\varepsilon$}
		\State Solving the $\h x$-subproblem (\ref{xsub-ADMMI}) via Algorithm \ref{subx:admm:Bes}.
		\State Computing $\h y^{k+1}$ via (\ref{ynew}).
		\State Updating $\h z^{k+1}$ via (\ref{zsub-ADMMI}).
		\EndWhile
	\end{algorithmic}
\end{algorithm}

\subsection{Global Convergence}		\label{convanaly}
 In contrast to the existing literature on the convergence analysis of ADMM or its variants \cite{LiPong15,HongLuoRazaviyayn16,WangYinZeng15,Tao20}, we
 proves it converges to a d-stationary point without the KL assumption. We define the merit function:
\begin{eqnarray}
\label{shMF}
{\cal T}(\h x,\h y)=\gamma\frac{\|\h x\|_1}{\|\h x\|_2}\!+\!\iota_{\cal X}(\h x)\!+\!\frac{1}{2}\|A{\h x}-{\h b}\|^2\!+\!\frac{\beta}{2}\|\h x-\h y\|_2^2,
\end{eqnarray}
and denote ${\cal T}^k:={\cal T}({\h x^k},{\h y^k})$ for succinctness.


\begin{lemma}\label{theo49}
		Let $\{\h w^k\}$ be the sequence generated by  ADMM$_p^+$. If
$\beta>2L$, then there exists a constant $c_1>0$ such that
			${\cal T}^{k+1}\le {\cal T}^k-c_1 \|\h y^k-\h y^{k+1}\|_2^2 .$
	\end{lemma}
\begin{proof}
First, it follows from the optimality condition of (\ref{ysub-ADMMI}) that ${\h z}^{k+1} = A^\top(A{\h y}^{k+1}-{\h b})$.
Then, it further implies that
\begin{eqnarray} \label{boundual} \|{\h z}^k-{\h z}^{k+1}\|_2\le L\|{\h y}^k-{\h y}^{k+1}\|_2, \end{eqnarray}
where $L=\sigma_{\max}(A^\top A)$ where $\sigma_{\max}(\cdot)$ represents the largest eigenvalue.
Next, invoking (\ref{xsub-ADMMI}), it leads to
${\cal L}_{\cal A} ({\h x}^{k+1},{\h y}^k,{\h z}^k)\le {\cal L}_{\cal A} ({\h x}^{k},{\h y}^k,{\h z}^k).$
Then, using (\ref{ysub-ADMMI}), it yields that
$$ {\cal L}_{\cal A}({\h x}^{k+1},{\h y}^{k+1},{\h z}^k) \le {\cal L}_{\cal A}({\h x}^{k+1},{\h y}^{k},{\h z}^k)-\frac{\beta}{2}\|{\h y}^k-{\h y}^{k+1}\|^2,$$
which is due to ${\cal L}_{\cal A}({\h x}^{k+1},{\h y},{\h z}^k)$ is strongly convex with respect to ${\h y}$ with strongly convex coefficient of
$\frac{\beta}{2}$.
In view of (\ref{zsub-ADMMI}), we obtain that
${\cal L}_{\cal A}({\h x}^{k+1},{\h y}^{k+1},{\h z}^{k+1}) = {\cal L}_{\cal A}({\h x}^{k+1},{\h y}^{k+1},{\h z}^k)+\frac{1}{\beta}\|{\h z}^k-{\h z}^{k+1}\|^2.$
Combining above three inequalities with (\ref{zsub-ADMMI}), we have that
\begin{align}\label{sixpsix} {\cal L}_{\cal A}({\h x}^{k+1},{\h y}^{k+1},{\h z}^{k+1})
\le {\cal L}_{\cal A}({\h x}^{k},{\h y}^{k},{\h z}^{k})-(\beta/2-L^2/\beta)\|{\h y}^k-
{\h y}^{k+1}\|^2. \end{align}
Next, we show that
\begin{eqnarray}\label{Teq} {\cal T}^{k+1} \le {\cal L}_{\cal A}({\h x}^{k+1},{\h y}^{k+1},{\h z}^{k+1}) +\frac{L}{2}\|{\h x}^{k+1}-{\h y}^{k+1}\|^2.\end{eqnarray}
Recall the definition ${\cal T}^{k+1}$ in (\ref{shMF}) and ${\cal L}_{\cal A}({\h x}^{k+1},{\h y}^{k+1},{\h z}^{k+1})$ in  (\ref{eq:augL}).
To show (\ref{Teq}),
we only need to prove that
\begin{eqnarray*} \frac{1}{2}\|A{\h x}^{k+1}-{\h b}\|^2 \le  \frac{1}{2}\|A{\h y}^{k+1}-{\h b}\|^2 + ({\h z}^{k+1})^\top ({\h x}^{k+1}-{\h y}^{k+1})+\frac{L}{2}\|{\h x}^{k+1}-{\h y}^{k+1}\|^2. \end{eqnarray*}
Invoking the optimality condition of (\ref{ysub-ADMMI}),  it leads to
${\h z}^{k+1}=\nabla (\frac{1}{2}\|A{\h y}^{k+1}-{\h b}\|^2)$.
By using this fact and $L=\sigma_{\max}(A^\top A)$, the above inequality
follows directly.
Consequently,
\begin{eqnarray}&&{\cal T}^{k+1}\le  {\cal L}_{\cal A}({\h x}^{k+1},{\h y}^{k+1},{\h z}^{k+1}) +\frac{L}{2}\|{\h x}^{k+1}-{\h y}^{k+1}\|^2\nn\\
&&\le {\cal L}_{\cal A}({\h x}^{k},{\h y}^{k},{\h z}^{k}) -\frac{3L}{8}\|{\h y}^k-{\h y}^{k+1}\|^2,\label{tmp1}
 \end{eqnarray}
 where the first inequality is due to (\ref{Teq}), the second is due to (\ref{zsub-ADMMI}), (\ref{sixpsix}), (\ref{boundual}) and $\beta>2L$. Next, we have that
${\cal T}^k\ge{\cal L}_{\cal A}({\h x}^{k},{\h y}^{k},{\h z}^{k}).$
Combining (\ref{tmp1}) with the above inequality, the assertion holds with $c_1=\frac{3L}{8}$.
\end{proof}

\begin{lemma}\label{lemma45}
	Let $\{\h w^k\}$ be the sequence generated by  ADMM$_p^+$. Then there exists  a constant $c_2>0$ such that
${\text{\rm dist}}(\h 0,\partial{\cal T}(\h x^{k+1},\h y^{k+1}))\le c_2\|\h y^{k+1}-\h y^k\|_2.$
\end{lemma}
\begin{proof} The proof  is similar to \cite[Lemma 5.7]{Tao20} and thus omitted. \end{proof}

Next, we present   the subsequential convergence of ADMM$_p^+$ under the boundedness of $\{\h x^k\}$ which
  is a standard assumption   to ensure  existence of
accumulation point \cite{ZengYuPong20,Attouch13}. The boundedness of $\{\h x^k\}$ can be guaranteed  by the boundedness of  the set of $\{{\h x}\in{\cal X}|F({\h x})\le F({\h x}^0)\}$ which
can be further ensured by  no nonnegative vectors in ${\cal N}(A)$.
The proof of the following theorem is standard  \cite{Tao20,LiPong15} and  thus omitted.

\begin{theorem}\label{theo2}
		Let $\{{\h w}^k\}$ be the sequence generated by   ADMM$_p^+$.	If $\{\h x^k\}$ is bounded and $\beta>2L$, we have the following	statements:
		\begin{itemize}
			\item[(i)] $\lim_{k\to\infty}\|{\h x}^k-{\h x}^{k+1}\|=0$, $\lim_{k\to\infty}\|{\h y}^k - \h y^{k+1}\|=0$,
and $\lim_{k\to\infty}\|\h z^k-\h z^{k+1}\|=0$;
			\item[(ii)] The sequence $\{{\h w}^k\}$  has at least one accumulation point ${\h w}^\infty$.

		\end{itemize}
	\end{theorem}
 Next, we show the global convergence of ADMM$_p^+$ to a d-stationary point by assuming $A^\top{\h b}
\not\le {\bf 0}$, the boundedness of  $\{{\h x}^k\}$ and $\beta$  sufficiently large. The first assumption is
to guarantee $\bf 0$ not being a accumulation point.
The latter two assumptions are usually imposed
for  the  convergence  \cite{Tao20,LiPong15,HongLuoRazaviyayn16}.

\begin{theorem}\label{thm:global} 
	Let  $\{\h w^k\}$ be the sequence generated by   ADMM$_p^+$.
If $A^\top{\h b}
\not\le {\bf 0}$, $\beta>2L,$ and $\{\h x^k\}$ is bounded, then (i) any accumulation point of $\{\h x^k\}$  is a d-stationary point  of (\ref{L1o2uncon}), (ii) $\{{\h w}^k\}$ has finite length, i.e.
$\sum_{k=1}^\infty\|{\h w}^{k+1}-{\h w}^k\|<\infty,$
and hence $\{{\h w}^k\}$ converges to  a stationary point ${\h w}^\infty:=({\h x}^\infty,{\h y}^\infty,{\h z}^\infty)$ satisfying
 \begin{eqnarray}
 \label{dtationaryP}
	\left\{\begin{array}{l}
	({\h x}-{\h x}^\infty)^\top\left( \gamma(\frac{{\bf 1}}{\|{\h x}^\infty\|_2}-\frac{\|{\h x}^\infty\|_1}{\|{\h x}^\infty\|_2}{\h x}^\infty)+ {\h z}^\infty \right)\ge 0\;\;\;\;\forall {\h x}\in{\cal X},\\
	A^\top(A {\h y}^\infty -{\h b}) -{\h z}^\infty=0,\\
	{\h x}^\infty = {\h y}^\infty.
	\end{array}\right.
\end{eqnarray}
\end{theorem}
\begin{proof}(i)
We first show that
 any accumulation point ${{\h x}^\infty}$ of the sequence $\{{\h x}^k\}$ generated by (\ref{ADMMschI}) cannot be ${\bf 0}$.
Suppose not.  Then, there exists a subsequence of $\{{\h w}^{k_j}\}$ converging to ${\h w}^\infty$  where
 ${\h x}^{k_j}\to{\h x}^\infty={\h 0}$. Thus, ${\h x}^{k_j+1}\rightarrow {\h x}^{\infty}={\bf 0}$ due to Theorem \ref{theo2}.
  Also, one has ${\h y}^{k_j}-\frac{1}{\beta}{\h z}^{k_j}\rightarrow{\bm \xi}^\infty:={\h y}^\infty-\frac{1}{\beta} {\h z}^\infty$.
 Next, we show that ${\bm\xi}^\infty\in{\mathbb R}^n_{-}$.
  In what follows,  we show the solution of (\ref{close}) is ${\bf 0}$ if and only if
 ${\h q}\in{\mathbb R}^n_{-}$ in (\ref{close}).
 For  the ``if" part, it is obviously true. For the ``only if" part,
i.e., if $\bf 0$ is a solution of (\ref{close}), then ${\h q}\in{\mathbb R}^n_{-}$.
Suppose not. Then, ${\h q}\not\in{\mathbb R}^n_{-}$. Thus, there exists at least one index (without loss of generality) $q_1>0$ and
$q_1\ge q_2\ge\cdots\ge q_n$. We define ${\h q}^+=\max({\h q},{\bf 0})$. Thus, ${\h q}^+ \neq{\bf 0}$. According to \cite[Theorem 3.2]{Tao20},
we see that the solution of
 \begin{eqnarray}\label{equproT} {\arg\min}_{{\h x}\in \mathbb{R}^{n}}\left( \frac{\|\h x\|_1}{\|\h x\|_2}+\frac{\rho}{2}\|{\h x}-{\h q}^+\|_2^2\right), \end{eqnarray}
cannot be ${\bf 0}$ since the solution of (\ref{equproT}) is at least one-sparse. It contradicts $\bf 0$ being a solution of (\ref{close}).

It follows from (\ref{xsub-ADMMI}) that
${\h x}^{k_j+1}\!\in\!\prox_{[({L}_1/{L}_2)^+/\rho]}({\h y}^{k_j}-\frac{1}{\beta}{\h z}^{k_j}).$
Taking $j\to\infty$  and  invoking  \cite[Theorem  1.25]{RockWets}, we have
${\h x}^\infty\in\prox_{[({L}_1/{L}_2)^+/\rho]}({\h y}^\infty-\frac{1}{\beta}{\h z}^\infty).$
Consequently, ${\bm\xi}^\infty\le {\bf 0}$ due to ${\h x}^{\infty}={\bf 0}$.
Since ${\h y}^{k_j}\rightarrow{\h y}^\infty= {\h 0}$ due to ${\h x}^\infty-{\h y}^\infty = {\bf 0}$,
 ${\h z}^{k_j}\rightarrow-\beta{\bm \xi^\infty}$.
Invoking ${\h z}^{k_j}=A^\top (A {\h y}^{k_j}-{\h b})$ and letting $j\to\infty$,
it leads to $A^\top {\h b}={\beta}{\bm \xi^\infty}$ which contradicts to $A^\top{\h b}
\not\le {\bf 0}$. Thus, ${\h x}^\infty\neq{\bf 0}$.

Next, we show any accumulation point ${\h x}^\infty$ of $\{\h x^k\}$  is a d-stationary point  of (\ref{L1o2uncon}) with ${\cal X}={\mathbb R}^n_+$.
The sequence of $\{{\h w}^k\}$ is bounded and hence it has a subsequence $\{{\h w}^{k_j}\}$ such that
 ${\h w}^{k_j}\to\h w^\infty$ as $j\to+\infty$.
 From the optimality condition of (\ref{ADMMschI}), we have
 \begin{eqnarray} \label{AlgoKKTAD}
	\left\{\begin{array}{l}
	({\h x}-{\h x}^{k+1})^\top\left(\gamma(\frac{{\bf 1}}{\|{\h x}^{k+1}\|_2}-\frac{\|{\h x}^{k+1}\|_1}{\|{\h x}^{k+1}\|_2}{\h x}^{k+1})+ {\h z}^k+\beta({\h x}^{k+1}-{\h y}^k) \right)\ge 0\;\;\;\;\forall {\h x}\in{\cal X},\\[0.2cm]
	A^\top(A {\h y}^{k+1}-{\h b}) -{\h z}^{k+1}=0,\\[0.2cm]
	\beta({\h x}^{k+1}-{\h y}^{k+1})+{\h z}^k-{\h z}^{k+1}=0.
	\end{array}\right. \nn\\
\end{eqnarray}
The above system is also true when $k:=k_j$. Note that ${\h w}^{k_j+1}\to {\h w}^\infty$ as $j\to+\infty$ due to Theorem \ref{theo2} and ${\h w}^{k_j}\to{\h w}^\infty$.
 Then, taking  limit on both sides of the system (\ref{AlgoKKTAD}) with $k:=k_j$,
 we have that ${\h w}^\infty$ is satisfying  (\ref{dtationaryP})  due to Theorem \ref{theo2} and ${{\h x}^\infty}\neq{\bf 0}$.
By eliminating ${\h y}^\infty$ and ${\h z}^\infty$ from (\ref{dtationaryP}), we have (\ref{viopt}) holds with ${\bar{\h x}}={\h x}^\infty$,
which implies that ${\h x}^\infty$ is a d-stationary point of (\ref{L1o2uncon}) with ${\cal X}={\mathbb R}^n_+$.

(ii)
According to \cite{BCR98}, if at least one of the two subanalytic functions   maps bounded sets to bounded sets, then their sum
is  subanalytic. Since $\frac{1}{2}\|A{\h x}-{\h b}\|_2^2$ is real analytic and maps bounded sets to bounded sets, and  the function $\displaystyle{\gamma\frac{\|\h x\|_1}{\|\h x\|_2}+\iota_{\mathbb R_+^n}}(\h x)+\iota_{\{{\h x}|\|{\h x}\|\ge\varepsilon\}}(\h x)$ is semianalytic (for any sufficiently small $\varepsilon>0$) \cite{ZengYuPong20}, then their sum
is also subanalytic.
Similarly,
the function ${\cal T}({\h x},{\h y})$ defined in (\ref{shMF}) is also subanalytic.
Furthermore, invoking  Lemma \ref{lemma45}, any accumulation point  $({\h x}^\infty,{\h y}^\infty)$  of $\{\h x^k,\h y^k\}$ generated from (\ref{ADMMschI}) satisfies
${\bf 0}\in (\partial_{\h x} {\cal T}({\h x}^\infty,{\h y}^\infty),\partial_{\h y} {\cal T}({\h x}^\infty,{\h y}^\infty))$ with ${\h x}^\infty = {\h y}^\infty.$
Thus, $({\h x}^\infty,{\h y}^\infty)$ can not be $({\bf 0,\bf 0})$.
Define ${\tilde {\cal U}}=\{\h u:=({\h x},{\h y})\in{\mathbb R}^n \times {\mathbb R}^n | \|{\h u}\|_2 \ge \epsilon\}$ with
$0<\epsilon<\frac{1}{2}\|{\h u}^\infty\|_2$.
 Invoking  \cite[Theorem 3.1]{BDL07}, the merit function ${\cal T}(\h x,\h y)|_{\tilde {\cal U}}$
  satisfies the KL property since ${\cal T}(\h x,\h y)|_{{\tilde {\cal U}}}$ is continuous and its domain is closed.
Therefore, ${\cal T}({\h x},{\h y})$ satisfies the KL property at the point $({\h x}^\infty,{\h x}^\infty)$.
The remaining proof is standard and similar to \cite[Theorem 4]{LiPong15}, thus omitted here.
\end{proof}

\section{Numerical results}\label{NumRes}
In this section, we compare  ADMM$_p^+$ with state-of-the-art methods in sparse recovery.
We focus on the sparse recovery problem with the compressive matrix is highly coherent, on which
$L_1$ minimization fails.
All these algorithms are implemented on MATLAB R2016a, and performed on a desktop with Windows 10 and an Intel Core i7-7600U CPU processor (2.80GH) with 16GB memory. The stopping criterion is as follows:
\begin{eqnarray}\label{StopC}
{\tt{ RelChg }}:=\frac{\left\|{\h x}^{k}-{\h x}^{k-1}\right\|_{2}}{\max \left\{\left\|{\h x}^{k-1}\right\|_{2}, 0.1\right\}}<{\tt{ Tol }} \mbox{ or}\; k_{\max }>5n.
 \end{eqnarray}
We set {\tt Tol} as
\begin{eqnarray*}
{\tt {Tol}} = \left\{\begin{array}{ll}
10^{-6}        &\text {if } \sigma=0, \\
0.01 * \sigma  &\text {if } \sigma>0,
\end{array}\right.
\end{eqnarray*}
where $\sigma$ is the variance of the noise $(\sigma=0$ means the noiseless case).
 Two types of sensing matrices are considered:
(I) Oversampled DCT.   $A = [{\bf a}_1,{\bf a}_2,\ldots,{\bf a}_n]\in {\mathbb R}^{m\times n}$
	with each column $ {\bf a}_j:=\frac{1}{\sqrt{m}}\cos \left( \frac{2\pi {\bf w}j}{F}\right)(j=1,\ldots,n),$
	where ${\bf w}\in{\mathbb R}^m$ is an uniformly distribution on $[0,1]$ random vector  and
	$F\in {\mathbb R}_+$ controls the coherence.
(II) Gaussian matrix.  $A$ is subject to ${\cal N}({\bf 0}, \Sigma)$  with
	the covariance matrix given by $\Sigma=\{(1-r)I_n(i=j)+r \}_{i,j}$
	with $1>r>0$.
We generate an $s$-sparse ground truth signal ${\h  x^*}=|{\bar{\h x}}|\in{\mathbb R}_+^n$ with each nonzero entry of ${\bar{\h x}}$ following a Gaussian normal distribution.
\subsection{Algorithmic behaviors}
In the literature, there are some efficient methods applicable to the model (\ref{L1o2uncon}) with ${\cal X}={\mathbb R}^n_+$,
 including General Iterative Shrinkage Thresholding (GIST) (\cite[Algorithm 2]{NG}, \cite[Algorithm 1]{GZLHY})
and monotone accelerated proximal gradient method (APG) with fixed stepsize (APG1) \cite[Algorithm  1]{LinLi15}, monotone APG with line search (APG2)  \cite[Algorithm  2]{LinLi15},
nonmonotone APG with fixed stepsize (APG3) \cite[Algorithm  3]{LinLi15}, nonmonotone APG with line search (APG4) \cite[Algorithm  4]{LinLi15}
and the smoothed $L_1/L_2$ approach (SOOT) proposed in \cite{Audrey15}.
For a fair comparison, we
 incorporate Algorithm \ref{subx:admm:Bes} in each algorithm for computing the proximal operator of $({L}_1/{L}_2)^+$.
We test on
two types of matrices (Gaussian matrix, oversampled DCT) with ground-truth signals of sparsity $15$.
The size of the sensing matrix is $128\times 1024$. We set $\gamma=0.001$ in (\ref{L1o2uncon}), and $\beta=0.025$ in ADMM$_p^+$.
According to the theoretical results in \cite{NG,GZLHY,LinLi15}, each of GIST,
 APG1, APG2, APG3 and APG4 clusters at a critical point.

 In Table \ref{Table6.4}, we test on two types of matrices with three different choices of initial points (the first two in the MATLAB scripts): (1) {\tt rand(n,1)};
 (2) {\tt abs(randn(n,1))}; (3) The solution of $L_1$ minimization.
 For each instance,  we run 20 trials  for all of these
   algorithms and record the average results.
 We report  the computing time in seconds (Time), the objective function value (Obj) and
 the relative error (RErr$:=\frac{\|{\h x}^k-{\h x}^{*}\|_2}{\|{\h x}^{*}\|_2}$) when the stopping criterion (\ref{StopC}) is satisfied.
   Data in this table show that ADMM$_p^+$ converges faster than the other comparing algorithms except
  the cases of $L_1$ solution  respectively under the Gaussian matrix and the oversampled DCT.
   For each  scenario, ADMM$_p^+$ always achieves the lowest quantity of RErr, and its performance is very robust to the choices of
   initial points. This advantage represents another advantage of the proposed ADMM$_p^+$  over the other comparing algorithms,
   such as GIST, various versions of APG and SOOT whose numerical performances are sensitive to the initial  points.
 In Figure \ref{smooth}, we depict RErr   with respect to iteration number from ADMM$_p^+$ with other comparing algorithms.
Each plot in   Figure \ref{smooth} corresponds to the two types of  initial points under two types of compressive matrices: the left is from
 {\tt rand(n,1)} under oversampled DCT matrix and  the right is from {\tt abs(randn(n,1))}  with Gaussian matrix.
 Clearly,  ADMM$_p^+$ converges much faster than the others and  always achieves the lowest quantity of RErr among these comparing algorithms for both cases.
\begin{table}
	\begin{center}
		{\scriptsize\caption{Average computation results generated from different initial points.}\label{Table6.4}}
		{ \scriptsize\vskip -2mm\begin{tabular}{cccccccc}
			\hline
			{}    &  ADMM$_p^+$    & GIST     & APG1    & APG2   & APG3    & APG4   &SOOT   \\
			\hline
			\multicolumn{8}{c}{Gaussian matrix,  initial point: {\tt  rand(n,1)}}\\
			\hline
Obj   &  {\bf 2.38e-03} &1.44e-01 & 2.72e-01 & 3.14e-02 &2.83e-01  &8.18e-03  & 7.99e+01 \\
Time  &  {\bf 0.23 }  & 7.17  &   3.64 &  31.92 &   2.25 &  14.89 &  10.25\\
RErr  &  {\bf 4.91e-05}  & 1.45e-02& 5.53e-03 & 6.18e-03 & 5.76e-03 & 2.64e-03 &2.83e-01 \\
			\hline
\multicolumn{8}{c}{Gaussian matrix,  initial point: {\tt abs(randn(n,1))}}\\
\hline
Obj  & {\bf 2.60e-03} & 1.27e-01   & 3.03e-01& 1.04e-01  & 3.17e-01 & 8.26e-03 &7.89e+01 \\
Time &  {\bf 0.22} &    6.89  &  3.86 &   55.48 &   2.35 &  15.83 &  10.39\\
RErr & {\bf 5.17e-05} & 1.36e-02   & 6.12e-03& 1.23e-02  & 6.26e-03 & 2.18e-03 & 2.83e-01 \\
\hline
\multicolumn{8}{c}{Gaussian matrix,  initial point: $L_1$ solution}\\
			\hline
Obj  & {\bf 3.52e-04} & 3.94e-04  & 3.94e-04  & 3.94e-04 &3.94e-04  &3.93e-04   &  3.82e-04 \\
Time &   0.15 &    0.071 &   0.031  &  0.050  &  {\bf 0.029} &0.033     &  0.63\\
RErr  & {\bf 7.19e-06} & 8.51e-05  & 8.51e-05  & 8.51e-05 &8.51e-05  &8.51e-05   &  8.43e-05 \\
\hline
			\multicolumn{8}{c}{Oversampled DCT, initial point: {\tt  rand(n,1)}}\\
			\hline
Obj & {\bf 1.64e-04}& 6.80e-04 &1.68e-04& 3.84e-04& 1.67e-04& 2.96e-04 &2.10e-03 \\
Time  & {\bf  0.68 }&   6.30 &   1.54 &  24.67&    1.02 &  52.81&    5.43\\
RErr  & {\bf 1.45e-05} &3.73e-02 &1.76e-03 &2.37e-02& 1.56e-03 &1.81e-02 &6.97e-02\\
\hline
	\multicolumn{8}{c}{Oversampled DCT,  initial point: {\tt abs(randn(n,1))}}\\
			\hline
Obj   & {\bf 1.64e-04}  & 6.87e-04 &  1.70e-04 & 3.94e-04  & 1.71e-04  &3.25e-04    &2.14e-03 \\
Time    & {\bf 0.68 }  & 6.40   &  1.44   & 24.53   & 0.82    & 53.20    & 5.09\\
RErr   & {\bf 1.45e-05}  & 3.77e-02 &  2.40e-03 & 2.43e-02  & 2.56e-03  & 2.08e-02   & 7.05e-02 \\
\hline
\multicolumn{8}{c}{Oversampled DCT,  initial point:  $L_1$ solution}\\
			\hline
Obj   & {\bf 1.64e-04}   & {\bf 1.64e-04}    & {\bf 1.64e-04}   & {\bf 1.64e-04 }  &{\bf 1.64e-04}  & {\bf 1.64e-04}   &{\bf 1.64e-04} \\
Time &  0.34    &  {\bf 0.016}     & 0.056     & 0.053     &0.038    & 0.030     & 0.28\\
RErr  & {\bf 1.01e-05}   & 2.28e-05    & 1.29e-05   & 1.57e-05   &1.29e-05  & 2.06e-05   & 2.17e-05 \\
\hline
		\end{tabular}}
	\end{center}
\end{table}

\begin{figure}[t]
	\vspace{0cm}\centering{\vspace{0cm}
		\begin{tabular}{cc}
		\includegraphics[scale = .4]{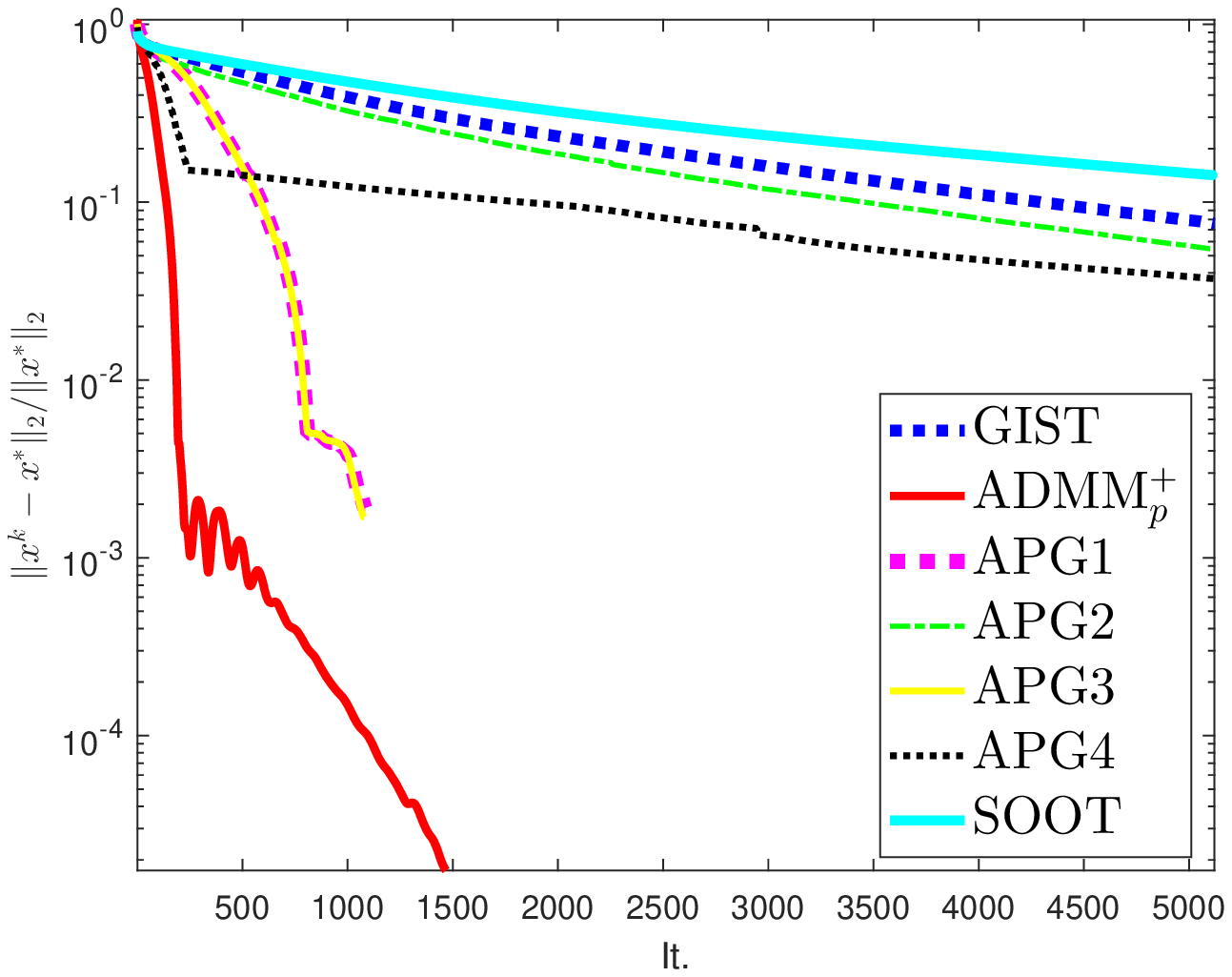}&\includegraphics[scale = .4]{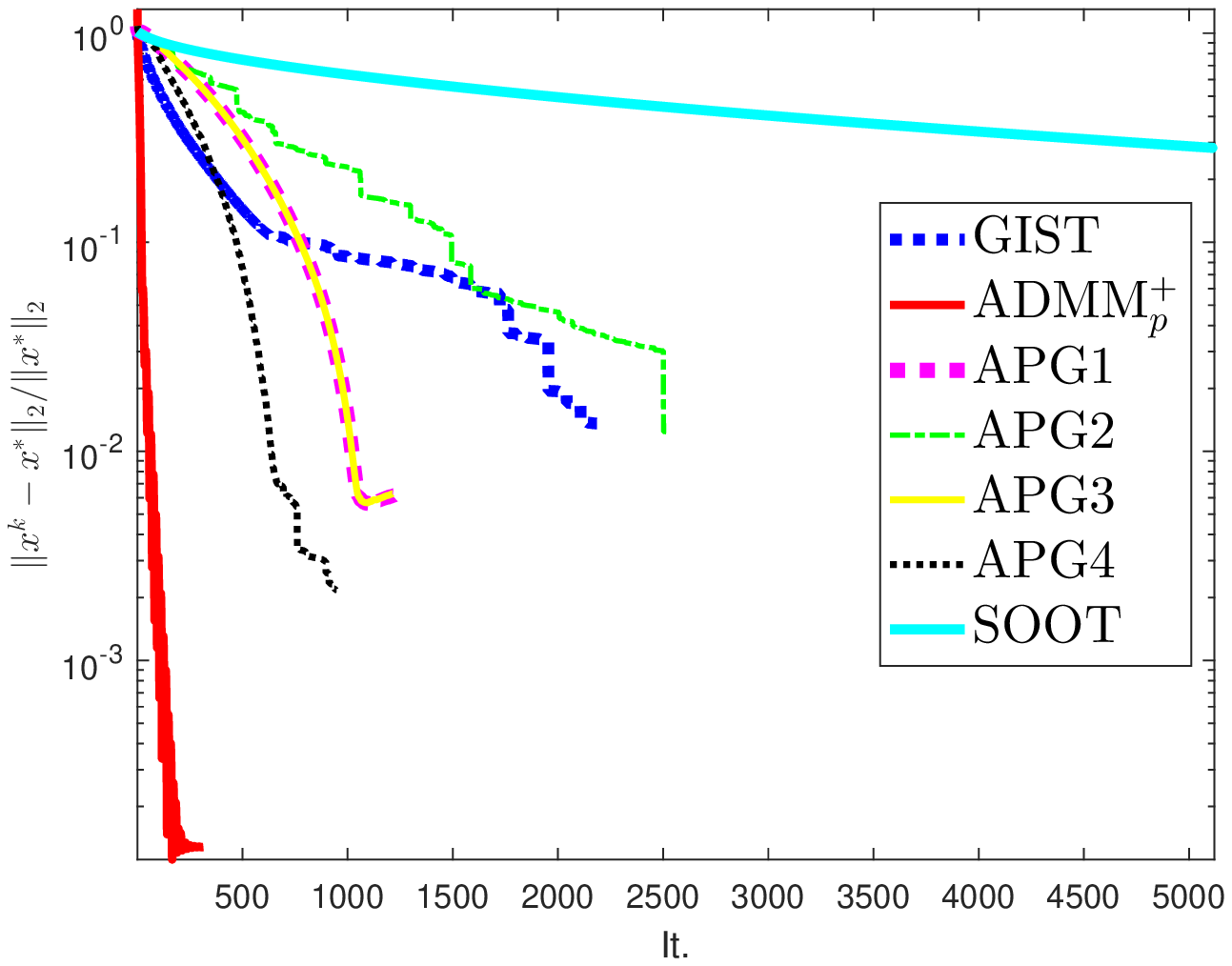}  \\
		\end{tabular}
	} \caption{The evolution of RErr with respect to the iteration number (It.): Initialized from {\tt rand(n,1)} under  oversampled DCT matrix (left), and from {\tt abs(randn(n,1))} under Gaussian matrix (right).
	}\label{smooth}\end{figure}

\begin{figure}[t]
	\vspace{0cm}\centering{\vspace{0cm}
		\begin{tabular}{cc}
			Success rates & Success rates\\
			\includegraphics[scale = .4]{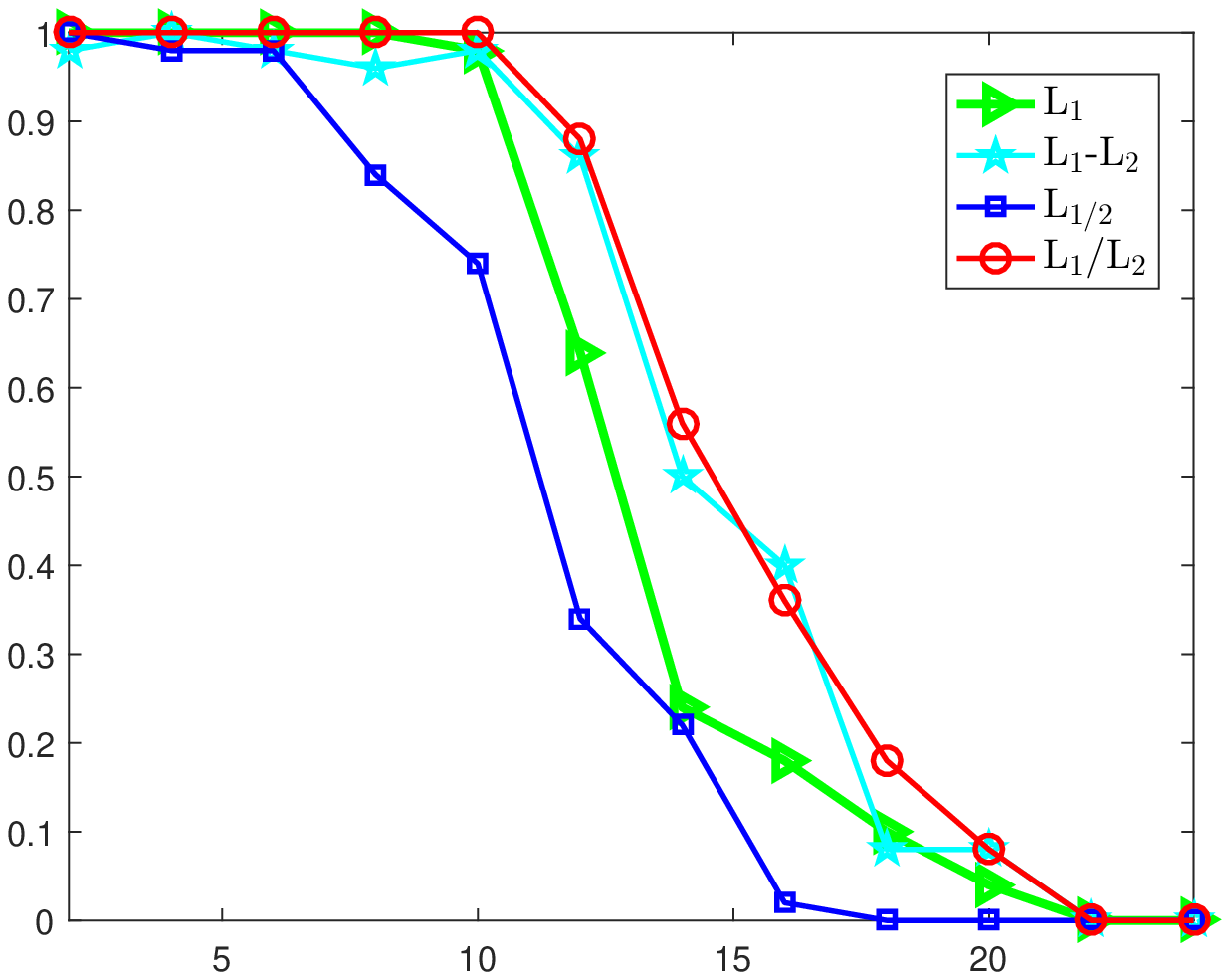}&\includegraphics[scale = .4]{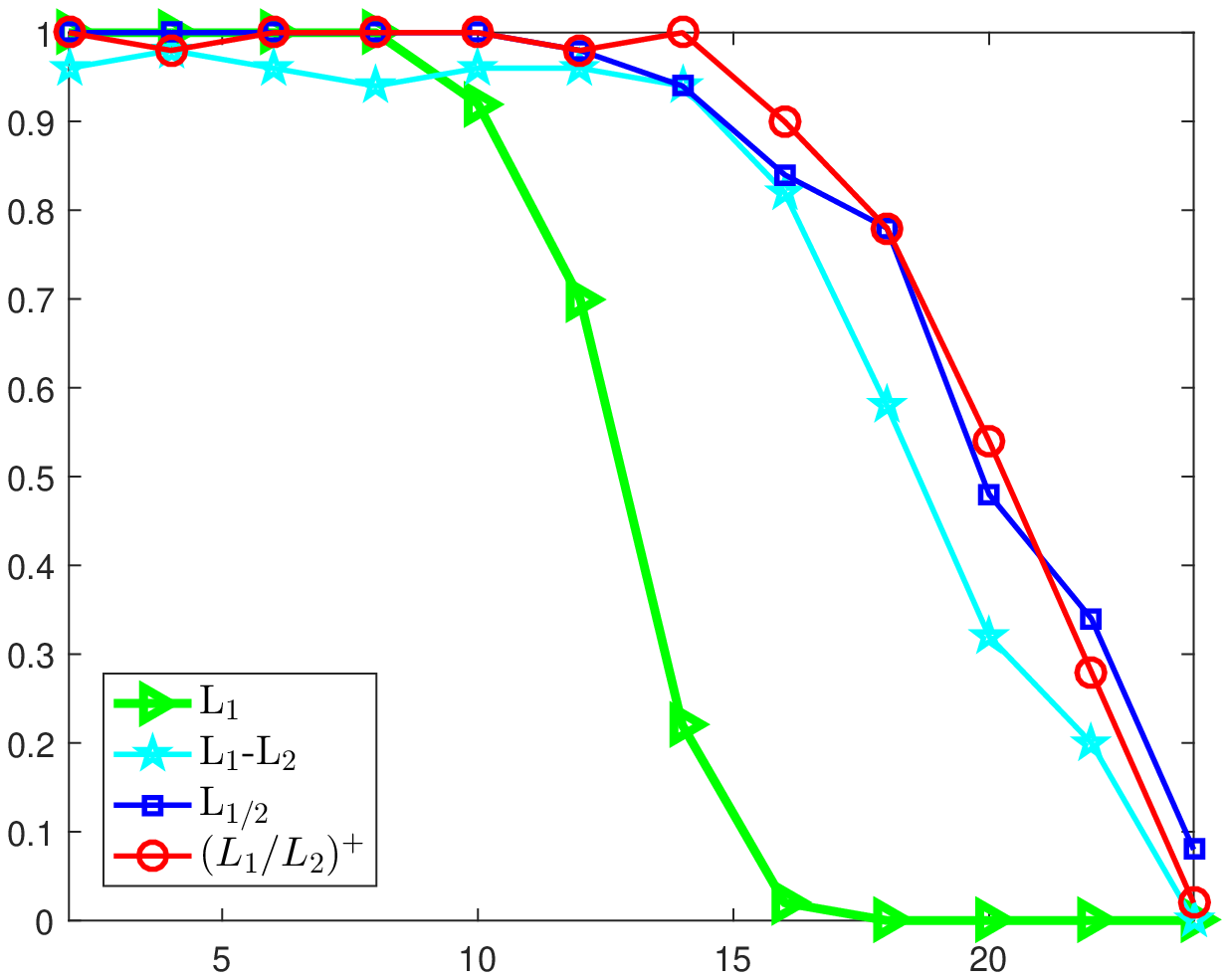}\\
           Algorithm failures & Algorithm failures\\
			\includegraphics[scale = .4]{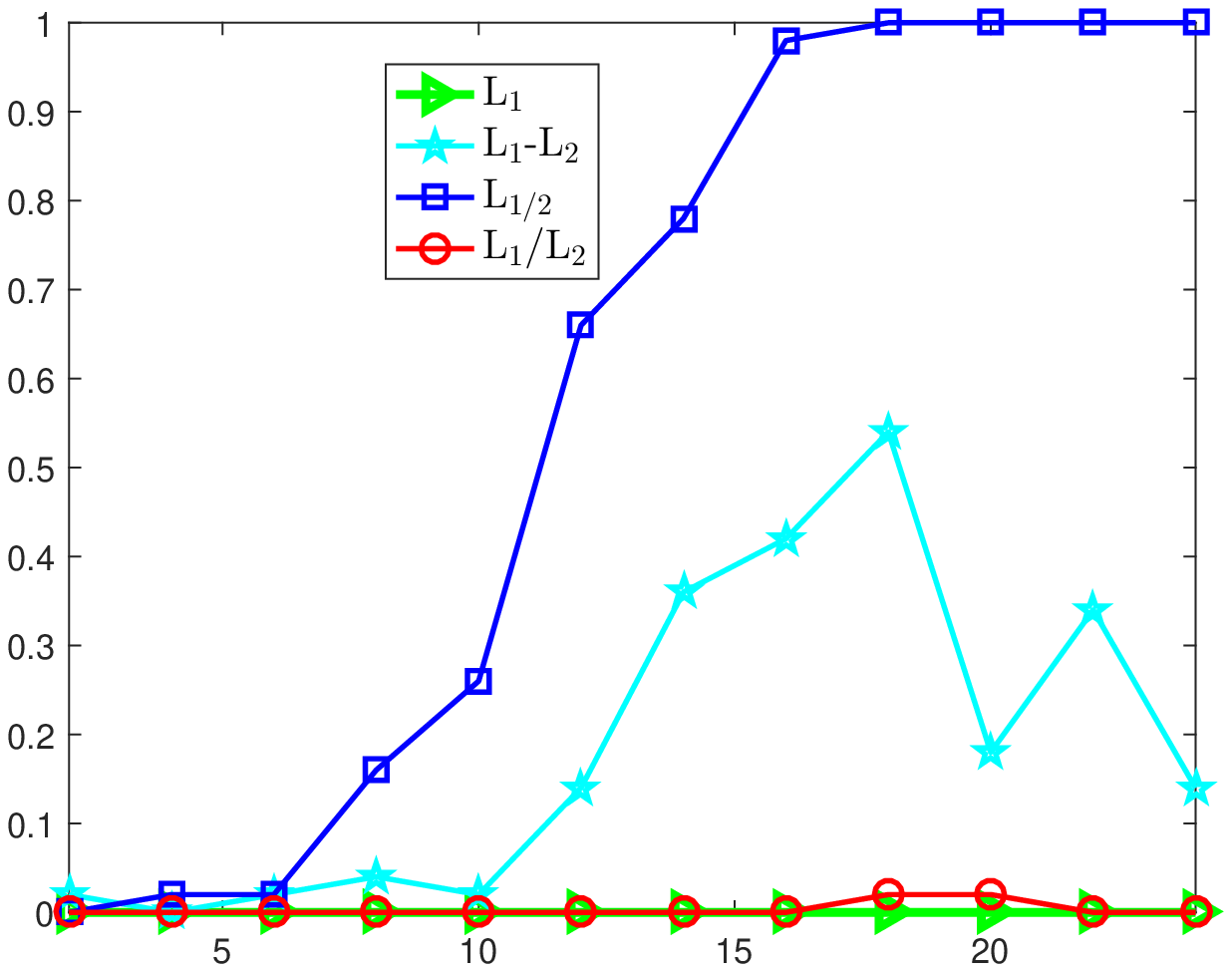}&\includegraphics[scale = .4]{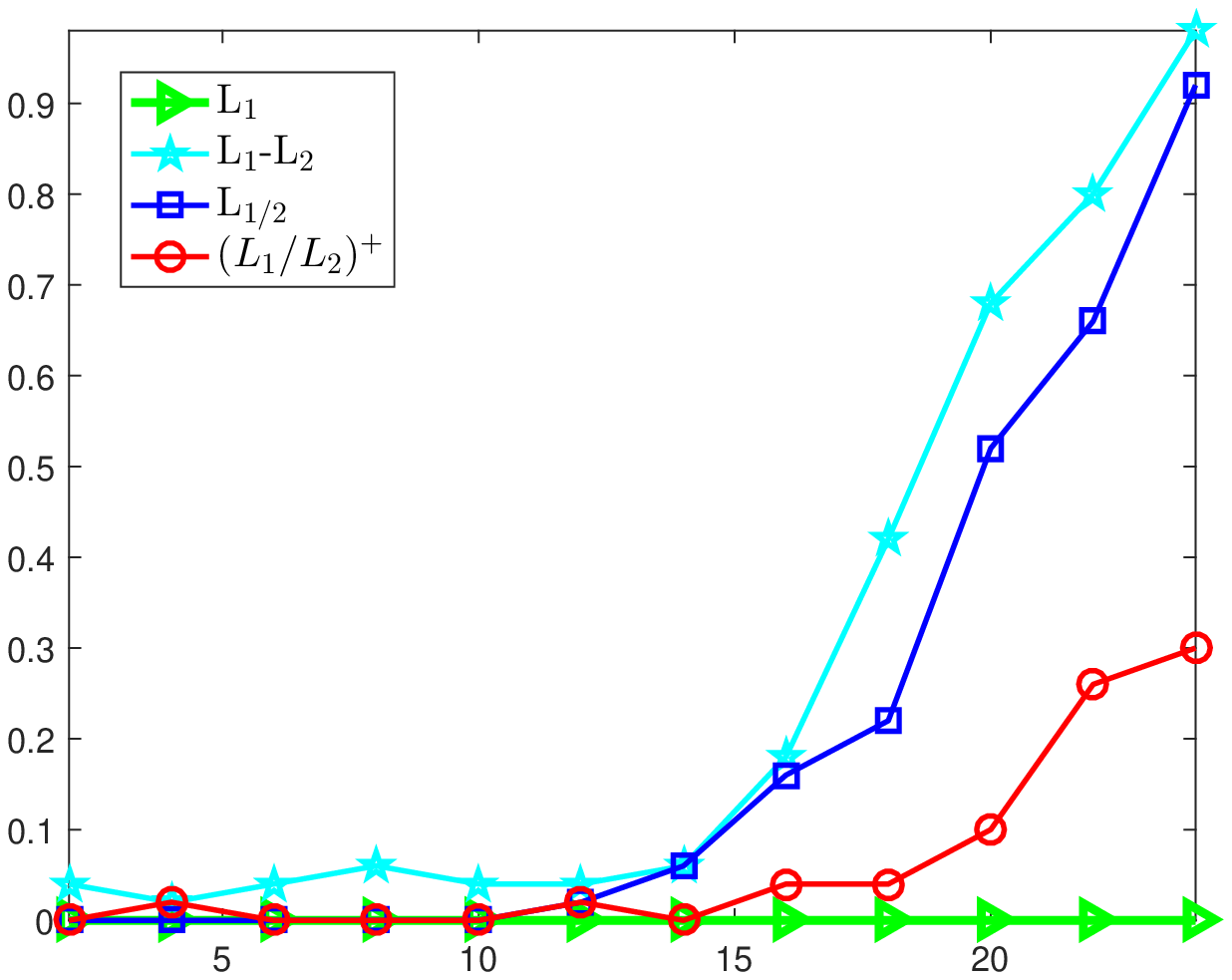}\\
Model failures & Model failures\\
			\includegraphics[scale = .4]{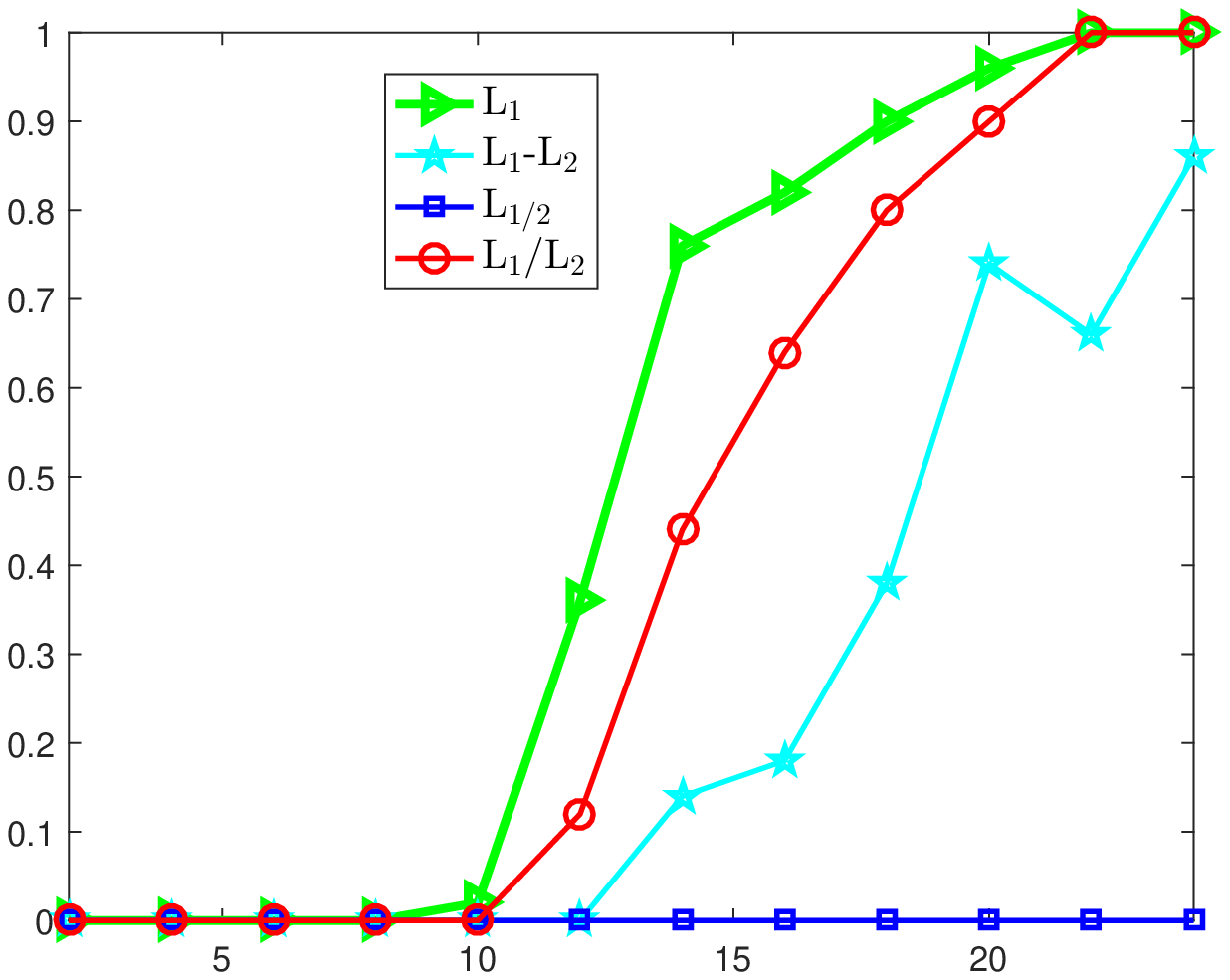}&\includegraphics[scale = .4]{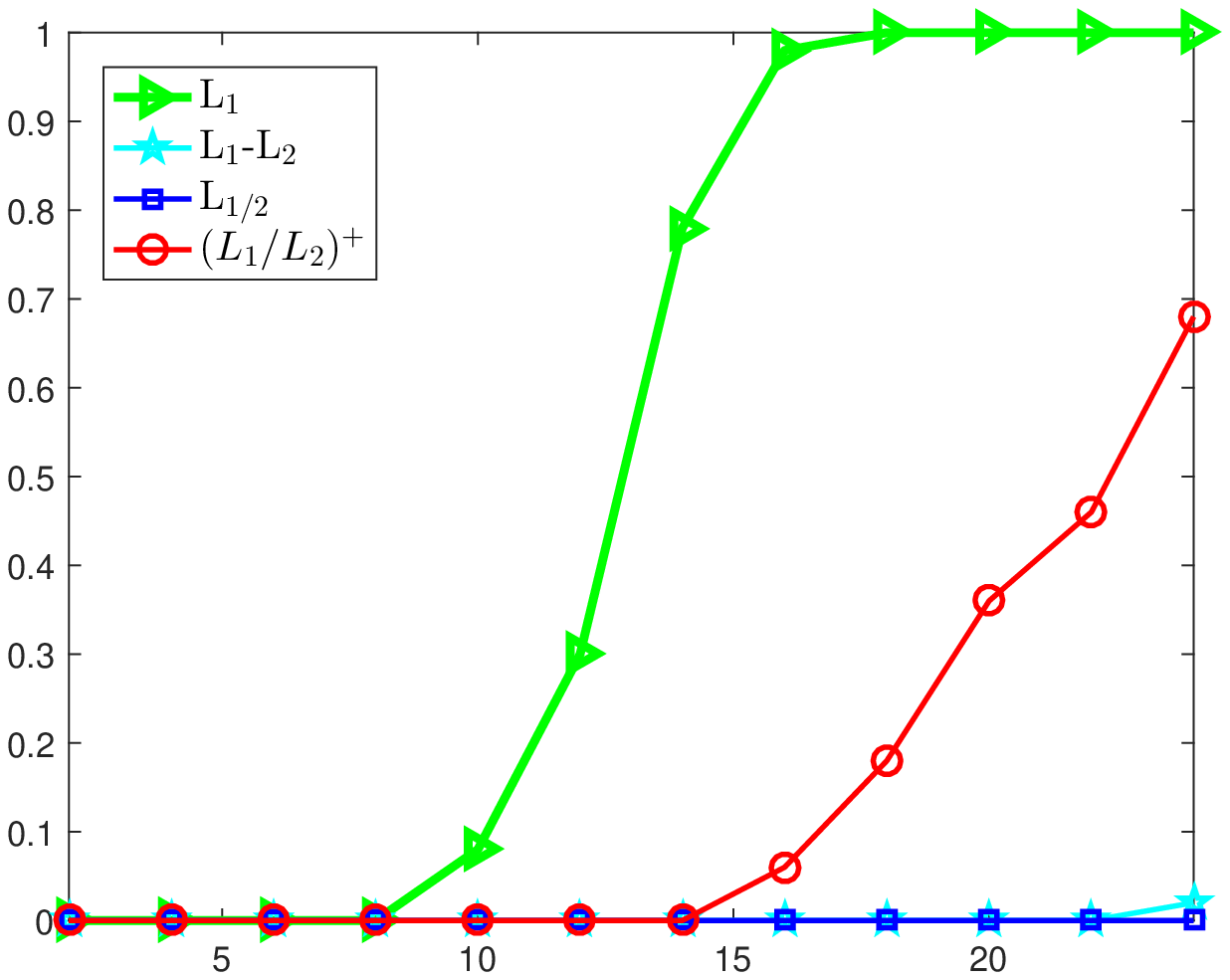}\\
		\end{tabular}	
	}
	\caption{Comparison results in the noisefree  case based on the oversampled DCT matrix with $F=10$ (left) and and  Gaussian matrix with $r=0.8$ (right). From top to bottom: success rates,  algorithm failures and model failures.}\label{suc}
\end{figure}
\subsection{Comparison on various models}
We show the efficiency of the proposed  ADMM$_p^+$ for $(L_1/L_2)^+$ minimization under the noiseless observation.  We compare  with other  sparse recovery unconstrained models: $L_1$,  $L_{1/2}$ \cite{CR07}, and $L_1
$-$L_2$ \cite{LY17}, all in an unconstrained formulation without nonnegative constraint.
We use the default setting  for each algorithm and unify their stopping criteria  as (\ref{StopC}) and
 set $\gamma=10^{-6}$ in all these models due to  the noisefree.
We consider over-sampled DCT matrix with $F=10$ and Gaussian matrix with $r=0.8$  of size $64\times 1024$, and the sparsity ranging from $2$ to $24$ with an increment of $2$.
The fidelity of sparse signal recovery is evaluated in terms of {\it success rate}, {\it model failure} and {\it algorithm failure} rates \cite{RWDL19,Tao20}. If  the relative error
of the reconstructed solution ${\hat{\h x}}$ to the ground truth $\h x^*$ is less than $10^{-3}$, we refer to it as a success.
{\it Success rate} is defined as  the number
of successes over the number of trials.
Furthermore,
we classify the failure of not recovery as model/algorithm failures by comparing the objective function $F(\cdot)$ at the ground
truth $\h {x^*}$ and the reconstructed solution ${\hat {\h x}}$.
If $F(\h x^*)<F({\hat {\h x}})$,  we refer it to as
{\it algorithm failure}. Otherwise,  we have {\it model failure}.
In Figure \ref{suc},  we present   success rate and  model/algorithm failure rates  for $(L_1/L_2)^+$,
$L_1$,  $L_{1/2}$ and $L_1
$-$L_2$ by randomly simulating 50 trials for each scenario and computing the average results. For the oversampled DCT case,
 $(L_1/L_2)^+$ achieves the highest
success rate.
For the Gaussian matrix case, $(L_1/L_2)^+$ exhibits a slightly better than $L_{1/2}$ when the sparsity is less than 22
and otherwise comparable to $L_{1/2}$ in terms of success rate.
Nevertheless, $(L_1/L_2)^+$ performs much better than $L_1$ and $L_1
$-$L_2$  for  Gaussian case regarding  success rate.
Based on model/algorithm failure rates in Figure \ref{suc}, we observe that the algorithm failure rates of $(L_1/L_2)^+$ are much lower  than that of $L_{1/2}$ and $L_1
$-$L_2$ for both cases and achieves the lowest (as well as the $L_1$) for the oversampled DCT case. The  model failure rates of $(L_1/L_2)^+$  rank  second for both cases and are always worse than $L_1$-$L_2$ and $L_{1/2}$ and better than $L_1$.
 These results
illustrate the efficiency of the ADMM$_p^+$ for both types of the sensing matrices and prompt us to further work on the model improvement of $(L_1/L_2)^+$ when
the sparsity level is increasing.
\subsection{Recovery of nonnegative signal from coherent dictionaries}
We illustrate  the efficiency of the ADMM$_p^+$ for solving  Examples 1-3 of \cite{YEX14} by comparing with the scaled gradient projection method (SGPM) which is a state-of-the-art algorithm in $(L_1/L_2)$ area  \cite{ELX13,YEX14}.   These examples are constructed to show the superiority of $(L_1/L_2)$ minimization over $L_1$ or $L_p$ ($0<p<1$) minimization.
 The SGPM is  forward-backward algorithm applied to (\ref{FBmodel}) by setting $g({\h x})=\iota_{{\mathbb R}^n_+}({\h x})$ and $h({\h x})=\gamma \frac{\|{\h x}\|_1}{\|{\h x}\|_{2}}+\frac{1}{2}\|A {\h x}-\mathbf{b}\|_{2}^{2}$ with  line research.
 All these three examples are linear systems, i.e., $A{\h x}={\h b}$ and denote the corresponding matrix
  as $A^{(i)}$, ${\h b}^{(i)}$  for each $i\in\{1,2,3\}$.
We test on Examples 1-3,  all these matrices $A^{(i)}\;(i=1,2,3)$ are defined with values of $n=50,100$ and $p=0.9,0.95$, and the vectors of ${\h b}^{(i)}\;(i=1,2,3)$
 are  of $n$ random numbers subjected to uniform distribution on $[0,1]$.
 The model parameter of $\gamma$ in (\ref{L1o2uncon}) is set to  $0.01$.

  For the SGPM, we use the defaulted setting as in \cite{ELX13,YEX14}, i.e., $\delta=1, c_{0}=10^{-9}, \;\xi_{1}=2,\; \xi_{2}=10$ and $\sigma=0.01$. Set $\beta=0.8$ in ADMM$_{p}^{+}$. For these three examples, we set the initial point as $x^{0}=0.05(100+0.01 \eta_i)$ and $\eta_i \thicksim N(0,1)$.
 In order to measure the extent of satisfying optimality condition of (\ref{viopt}), we define the Karush-Kuhn-Tucker (KKT) residual on  the  support set (KKT$_R$) of the  last iterate ${\hat {\h x}}$ as:
 \begin{eqnarray*} \text{KKT$_R$} = \left\| \gamma\left(\frac{{\text{sign}}({\hat{\h x}}_{\hat{\Lambda}})}{\|{\hat{\h x}}\|_2}-\frac{\|{\hat{\h x}}\|_1}{\|{\hat{\h x}}\|_2^3}{\hat {\h x}}_{\hat{\Lambda}}\right)+(A_{\hat{\Lambda}})^\top(A{\hat{\h x}}-{\h b})\right\|_2,   \end{eqnarray*}
where ${\hat\Lambda}=\operatorname{supp}(\hat{\h {x}})$.
In Table \ref{Table1.1}, we record the results of ADMM$_p^+$  and SGPM  in terms of final objective function value of (\ref{L1o2uncon}) (Obj),
the KKT$_R$ and computational time in seconds (Time).  Table \ref{Table1.1} clearly shows that ADMM$_{p}^{+}$ performs much better than SGPM in terms of achieving much lower objective function values, ending up with higher accuracy while taking less time.
 For the scenario of $n=50$ and $p=0.9$ of Example 2, we find that ADMM$_{p}^{+}$ recovers the one-sparse solution ${\h x}^{(2)}=[2,0, \cdots, 0]^{\top}$ while SGPM does not.

\begin{table}[h]
		\caption{Comparison  between ADMM$_p^+$  and SGPM   on  Examples 1, 2 and 3 via solving (\ref{L1o2uncon}).}\label{Table1.1}
		\begin{center}\vspace{-0.cm}{\begin{tabular}{ccccccc}
			\hline
			(Ex. ,$n$, $p$)  & \multicolumn{2}{c}{Obj} &\multicolumn{2}{c}{KKT$_R$}& \multicolumn{2}{c}{Time} \\
\cline{2-7}
                         & SGPM &   ADMM$_p^+$  & SGPM &  ADMM$_p^+$    & SGPM & ADMM$_p^+$\\
			\hline
             (1,50,0.95)   &5.64&{\bf 0.011}  &3.35& ${\bf 2.18\times10^{-4}}$ & 0.63  &  {\bf 0.33}\\
             (1,50,0.9)    &5.63&{\bf 0.010}  &3.35& ${\bf 4.38\times10^{-4}}$ & 0.64  &  {\bf 0.25}\\
             (1,100,0.95)  &11.2&{\bf 0.010}  &4.73& ${\bf 1.22\times10^{-4}}$ & 2.11  &  {\bf 0.97}\\
             (1,100,0.9)   &11.2&{\bf 0.010}  &4.73& ${\bf 2.69\times10^{-4}}$ & 2.02  &  {\bf 0.81}\\
             \hline
             (2,50,0.95)  &5.64&{\bf 0.010}  &3.35& ${\bf 5.69\times10^{-6}}$ & 0.61   & {\bf 0.45}\\
             (2,50,0.9)   &5.63&{\bf 0.010}  &3.35& ${\bf 1.05\times10^{-5}}$ & 0.59   & {\bf 0.16}\\
             (2,100,0.95) &11.2&{\bf 0.010}  &4.73& ${\bf 9.92\times10^{-6}}$ & 1.72   & {\bf 0.92} \\

             (2,100,0.9)  &11.2&{\bf 0.010}  &4.72& ${\bf 1.18\times10^{-5}}$ & 1.98   & {\bf 0.27}\\
             \hline
             (3,50,0.95)  & 3.00&{\bf 0.084}  &2.42& ${\bf 2.56\times10^{-3}}$  & 0.61 & {\bf 0.34}\\
             (3,50,0.9)   & 2.98&{\bf 0.083}  &2.41& ${\bf 1.97\times10^{-3}}$  & 0.66 & {\bf 0.38} \\
             (3,100,0.95) & 5.81&{\bf 0.11}   &3.38& ${\bf 1.24\times10^{-3}}$  & 1.81 & {\bf 0.95}\\
             (3,100,0.9)  & 5.77&{\bf 0.11}   &3.37& ${\bf 7.00\times10^{-4}}$ & 1.88  &  {\bf 0.95}\\
			\hline
		\end{tabular}}
	\end{center}
\end{table}
\subsection{DOAS}  We consider the wavelength misalignment problem in different optical absorption spectroscopy analysis (DOAS).
More specifically,
$
{\bm J}(\lambda)=\sum_{j=1}^{M} a_{j} {\h y}_{j}$ $\left(\lambda+v_{j}(\lambda)\right)+{\bm\eta}(\lambda).
$
 ${\bm J}(\lambda)$ presents the data and ${\h y}_{j}\left(\lambda+v_{j}(\lambda)\right)$ denotes the reference spectra
at the deformed wavelength $\lambda+v_{j}(\lambda)$ where $v_j(\cdot)$  denotes the deformations. The noise $\bm{\eta}(\lambda)$ are given at the wavelength $\lambda$ and $\left\{a_{j}\right\}_{j=1}^M$ are coefficients.

 In our experiments,    we  generate a dictionary
 for three  reference gases ($M=3$): HONO, NO2 and $\mathrm{O} 3$, and then deform each with a set of linear functions,
  i.e., $v_{j}(\lambda)=p_{j} \lambda+q_{j}$. We use $B_{j}$ $(j=1, \cdots, M)$ to denote a matrix with each column being deformed basis, i.e.. ${\h y}_{j}\left(\lambda+p_{k} \lambda+q_{\ell}\right)$ $(k=1, \cdots, K;\; \ell=1, \cdots, L)$
and
$
 {\h y}_{j} \in {\mathbb R}^{1024};\; p_{k}=-1.01+0.01 k,\; q_{\ell}=-1.1+0.1\ell.
$
By setting $K=L=21$, there is a total of 441 linearly deformed references for each of the three groups.

We generate the dictionary by imitating the relative magnitudes of a real DOAS dataset \cite{FP00} with normalization to the dictionary.
Then, to generate the data ${a}_j$, we  randomly pick up one entry with random magnitudes whose mean values are $1, 0.1, 2$ for HONO, NO2 and O3, respectively. Finally, the synthetic data ${\cal J}(\lambda)$ is generated by adding zero-mean Gaussian noise.
We test five different noise levels: ${\tt std}=0,1e-3,5e-3,1e-2,5e-2$.


 We solve the wavelength misalignment by considering the following model:
\begin{eqnarray} \label{DOASP}
 \min_{\{{\h x}_{j}\}_j} \frac{1}{2}\left\|\boldsymbol{J}-\left[B_{1}, \cdots, B_{M}\right]\left[\begin{array}{c}
\mathbf{x}_{1} \\
\vdots \\
\mathbf{x}_{M}
\end{array}\right]\right\|^{2}\!+\!\gamma \!\sum_{j=1}^{M} R\left(\mathbf{x}_{j}\right),
\end{eqnarray}
where $R(\cdot)$ represents the regularization function, and ${\h x}_j\in\mathbb R^{441}$ (${j=1,2,3}$).
We test (\ref{DOASP}) on different regularization functions to enforce sparsity. In particular, we set $R({\h x})=\iota_{\mathbb R^n_+}(\h x)$, $\|{\h x}\|_1 + \iota_{\mathbb R^n_+}(\h x)$,  $\|{\h x}\|_{1}-\|{\h x}\|_2$, $\|{\h x}\|^{1/2}$, $\frac{\|{\h x}\|_{1}}{\|{\h x}\|_2} + \iota_{\mathbb R^n_+}(\h x)$ in (\ref{DOASP}), respectively.
 We refer to these models as non-negative least square (NNLS), non-negative unconstrained $L_1$ (NNL1), $L_1$-$L_2$, $L_{1/2}$, $(L_1/L_2)^+$.
 For (\ref{DOASP}) with $(L_1/L_2)^+$ regularizer, we adopt ADMM$_p^+$ and SGPM to solve it.
 For  $L_1$-$L_2$, we use Algorithm 1 in \cite{LouOX15}.
 For NNLS, we use MATLAB's {\tt lsqnonneg} function.
 As for NNL1, we solve it by ADMM and
 for $L_{1/2}$, we solve it by \cite{LXY13}. For all these methods, we use the default setting.

 Tables \ref{Table1.2} and \ref{Table1.33} show the errors (${\tt err}=\|{\hat{\h x}}-{\h x}^*\|_2$) between the reconstructed vectors and the ground-truth, and
 computational time (Time (s)) under different  amounts of noise, respectively.  Each recorded value is the average of $20$ random realizations.
 In Fig. 2,  the ground truth  and  the error vectors of these comparing algorithms defined by the constructed signals minus the true signal  are presented in Plots (a) and (b),
 for the  scenario of ${\tt std}=0.05$. The horizontal heavy yellow line  surrounded by  cyan is caused by the full-dimension error vectors from $L_1$-$L_2$, SGPM, NNL1, $L_{1/2}$.
 The deviation of ADMM$_p^+$ is much smaller than the others.
  ADMM$_p^+$ achieves the best recovery quality in the sense of highest accuracy and sparsity.
 All the results demonstrate that ADMM$_p^+$ is comparable to
 NNLS in terms of accuracy for noiseless data, and even more  accurate than NNLS for noisy cases.
 In comparison with NNL1, $L_1$-$L_2$, ADMM$_p^+$ also ends up with much higher accuracy
 and takes less time. In contrast with $L_{1/2}$, ADMM$_p^+$ converges to a much more accurate solution while consuming a bit more
 time.
 Besides, for solving the same $(L_1/L_2)^+$ model, ADMM$_p^+$ costs significantly less time than
 SGPM while still achieving a much more accurate solution. More specificially, ADMM$_p^+$ reduces computational time by about $95\%\sim99\%$ compared to SGPM.

\begin{table}[t]
	\begin{center}
		\caption{Reconstructed error (${\tt err}=\|{\hat{\h x}}-{\h x}^*\|_2$) for DOAS.
}\label{Table1.2}
		\vspace{-0cm}\begin{tabular}{ccccccc}
			\hline
			{\tt std} \!\!&\!\! NNLS & NNL1 & $L_1$-$L_2$& $L_{1/2}$&\multicolumn{2}{c}{$(L_1/L_2)^+$}\\
			\cline{6-7}
               \!\!&\!\!    &     &           &          & ADMM$_p^+$&SGPM \\
                \hline
                 0     \!\!&\!\!  7.72e-16   & 2.70e-03    &   5.21e-05  &  3.60e-03   & {\bf 2.58e-05}&4.80e-02\\
                0.001  \!\!&\!\!  4.91e-03   & 7.95e-03    &   8.76e-04   &  2.41e-02  & {\bf 4.37e-04}&7.60e-02\\
                0.005  \!\!&\!\!  3.26e-02   & 1.92e-02    &   3.05e-03   &  6.95e-02  & {\bf 2.03e-03}&3.75e-01\\
                0.01  \!\!&\!\!  1.46e-01   & 1.61e-01    &   4.90e-03   &  1.04e-01   & {\bf 4.25e-03}&3.84e-01\\
                0.05  \!\!&\!\!  1.73e-01   & 1.75e-01    &   2.39e-02   &  1.30e-01   & {\bf 2.00e-02}&5.67e-01\\
                \hline
		\end{tabular}
	\end{center}
\end{table}

\begin{table}[t]
	\begin{center}
		\caption{Computational time (s) for DOAS under different noisy level.
}\label{Table1.33}
		\vspace{-0cm}\begin{tabular}{ccccccr}
			\hline
			{\tt std} \!\!&\!\! NNLS & NNL1 & $L_1$-$L_2$& $L_{1/2}$&\multicolumn{2}{c}{$(L_1/L_2)^+$}\\
          \cline{6-7}
               \!\!&\!\!    &     &           &          & ADMM$_p^+$ & SGPM \\
                \hline
                0    \!\!&\!\! 0.021   & 8.05     &   13.30  &  0.10    & 5.66 & 3110.00\\
                0.001 \!\!&\!\! 0.047   & 8.19    &   35.60  &  0.11   & 6.59 & 3030.00\\
                0.005 \!\!&\!\! 0.016   & 7.77    &   28.10  &  0.13   & 7.17 & 437.00\\
                0.01 \!\!&\!\! 0.057   & 8.50     &   33.90  &  0.13    & 8.34 & 167.00\\
                0.05 \!\!&\!\! 0.052   & 8.82     &   70.90  &  0.13    & 4.00 & 276.00\\
			\hline
		\end{tabular}
	\end{center}
\end{table}

\begin{figure}[htbp]
\vspace{0cm}\centering{\vspace{0cm}
\includegraphics[scale=0.3]{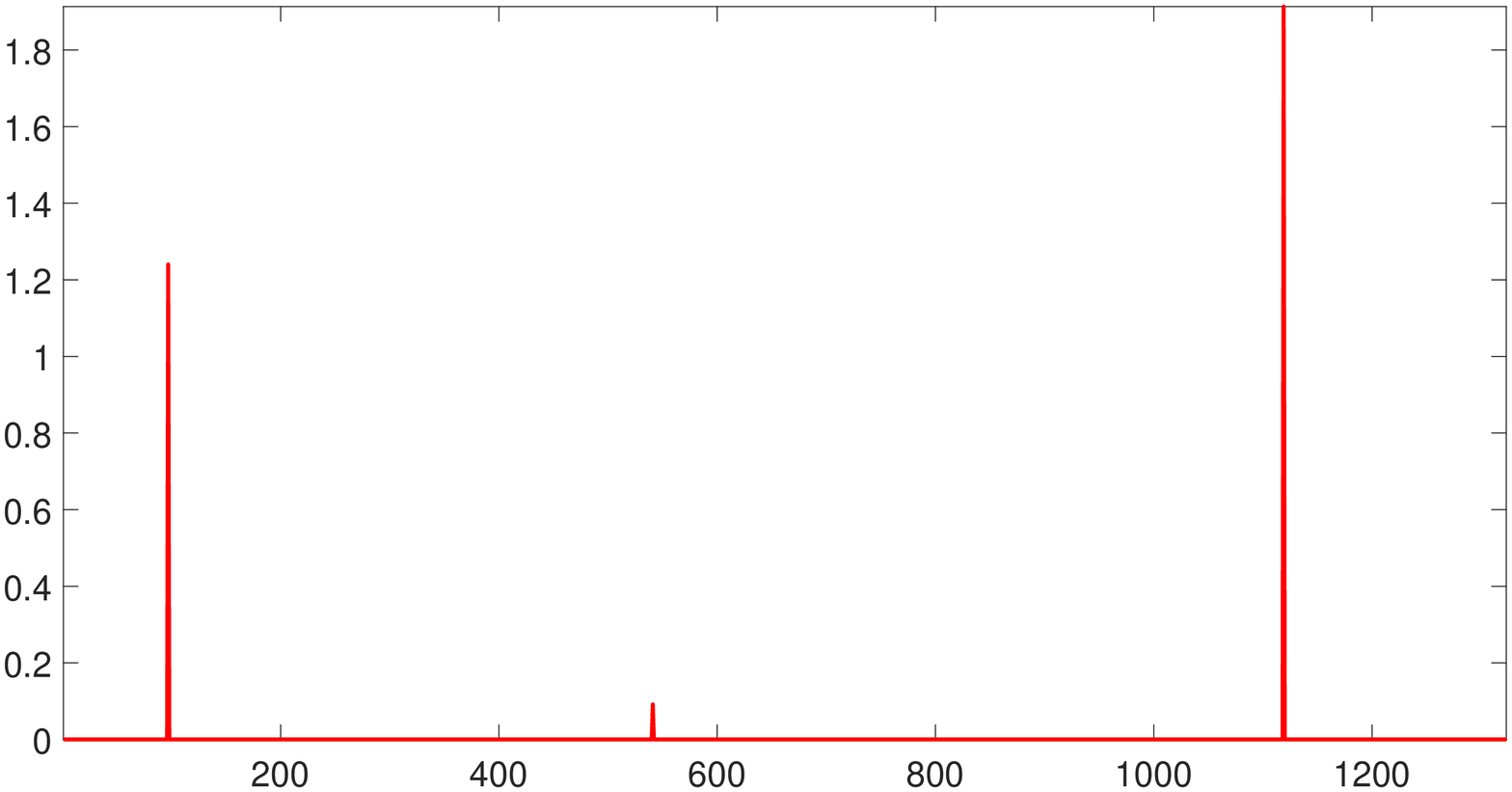}
\includegraphics[scale=0.3]{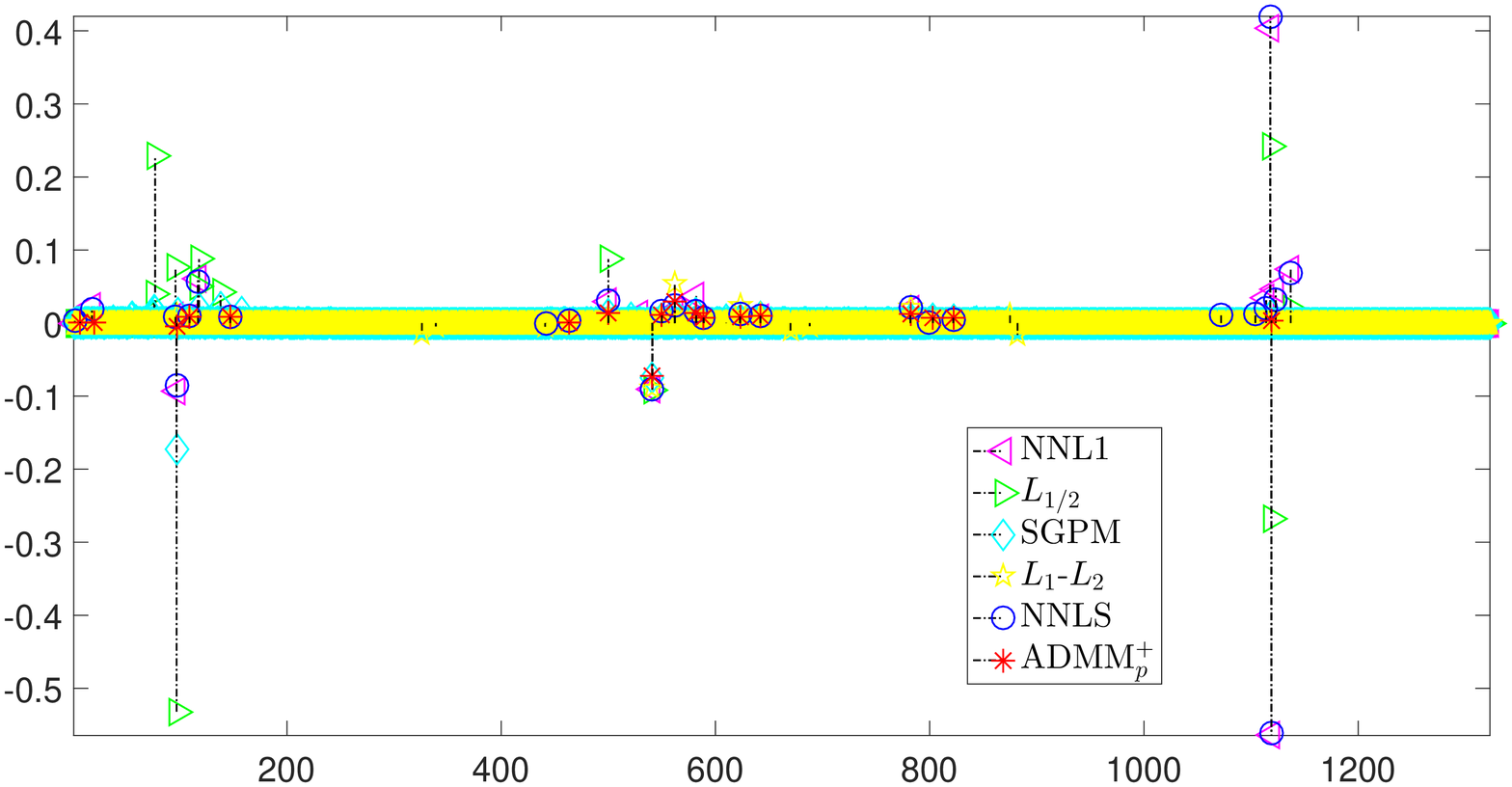}
} \caption{Comparison results for NNLS, NNL1, $L_1$-$L_2$, $L_{1/2}$, SGPM and ADMM$_p^+$  on DOAS data with additive noise (${\tt std}=5e-2$).
 (a) The ground truth with sparsity $3$ (top); (b)
 The error vectors (ERR$={\hat{\h x}}-{\h x}^*$) from these comparing algorithms, and
 its nonzero numbers of these error vectors are $27$,$\ 1323$,$\ 1323$,$\ 1323$,$\ 1323$,$\ 18$, respectively (bottom).
}\label{DOAS}\end{figure}

\section{Conclusions} \label{Sec-Conclusion}
 We carry out a unified theoretical study on  both $L_1/L_2$ minimization models, including
 the constrained and the unconstrained.
First, we prove that the existence of the globally optimal solution can be guaranteed by
   the  $\mu$-spherical section property of the null space of the matrix $A$.
   Second, we analyze the sparsity property of the constrained and the unconstrained models.
Third, we derive a closed-form solution of the proximal
   operator of $\left(L_{1} / L_{2}\right)^{+}$.
     Equipped with this, we propose  a specific splitting scheme (ADMM$_{p}^{+}$)
to solve the unconstrained $\left(L_{1}/L_{2}\right)^{+}$ model.
We establish its global convergence to a d-stationary solution
by verifying the KL property of the merit function.
Numerical simulations validate our analyses and demonstrate
that ADMM$_{p}^{+}$ outperforms  other state-of-the-art methods in sparse recovery.\\

\noindent{\bf Funding}
Min Tao was  partially supported by National Key Research and Development Program of China (2018AAA0101100), the Natural Science Foundation of China (No. 11971228) and  Jiangsu University QingLan Project. The work of Xiao-Ping Zhang is supported by the Natural Sciences and Engineering
Research Council of Canada (NSERC), Grant No. RGPIN-2020-04661.\\
\noindent{\bf Data Availibility}
The datasets generated during and/or analysed during the current study are available from
the corresponding author on reasonable request.\\
\noindent{\bf Declarations}
Conflict of interests The authors have no relevant financial or non-financial interests to disclose.

\end{document}